\documentclass[11pt,american,reqno]{extarticle}
\usepackage{amsmath}
\usepackage{amsthm}
\usepackage{libertineRoman}
\usepackage{helvet}
\usepackage{courier}
\usepackage[libertine]{newtxmath}
\usepackage[T1]{fontenc}
\usepackage{textcomp}
\usepackage[utf8]{inputenc}
\usepackage{geometry}
\usepackage{color}
\usepackage{babel}
\usepackage{csquotes}
\usepackage{mathtools}
\usepackage{enumitem}
\usepackage{esint}
\usepackage{microtype}
\usepackage[pdfusetitle,
 bookmarks=true,bookmarksnumbered=false,bookmarksopen=false,
 breaklinks=true,pdfborder={0 0 0},pdfborderstyle={},backref=false,colorlinks=true]
 {hyperref}
\hypersetup{
 allcolors=blue}
 
 \usepackage{aligned-overset}

\numberwithin{equation}{section}
\numberwithin{figure}{section}
\numberwithin{table}{section}
\usepackage{zref-clever}
\zcsetup{
    abbrev=true, %
    cap=true,
    nameinlink=false,
	rangetopair=false,
	endrange=stripprefix,
	sort=false,
	reffont=\upshape,
	lastsep={, and },
	comp=false
    }

\zcRefTypeSetup{equation}{
Name-sg = ,
name-sg = ,
Name-pl = ,
name-pl = ,
refbounds = {\textup{(},,,\textup{)}},
+refbounds-rb = {\textup{(},,,},
+refbounds-re = {,,,\textup{)}},
rangesep = {--}
}
\zcRefTypeSetup{item}{
Name-sg = ,
name-sg = ,
Name-pl = ,
name-pl = ,
refbounds = {\textup{(},,,\textup{)}},
+refbounds-rb = {\textup{(},,,},
+refbounds-re = {,,,\textup{)}},
rangesep = {--}
}
\zcRefTypeSetup{thmstepnvi}{
Name-sg = Step,
name-sg = step,
Name-pl = Steps,
name-pl = steps
}
\zcsetup{countertype={thm=theorem}}
\AddToHook{env/rem/begin}{\zcsetup{countertype={thm=remark}}}
\AddToHook{env/prop/begin}{\zcsetup{countertype={thm=proposition}}}
\AddToHook{env/lem/begin}{\zcsetup{countertype={thm=lemma}}}
\AddToHook{env/defn/begin}{\zcsetup{countertype={thm=definition}}}
\AddToHook{env/cor/begin}{\zcsetup{countertype={thm=corollary}}}
\theoremstyle{plain}
\newtheorem{thm}{\protect\theoremname}[section]
\theoremstyle{plain}
\newtheorem{cor}[thm]{\protect\corollaryname}
\theoremstyle{plain}
\newtheorem{prop}[thm]{\protect\propositionname}
\theoremstyle{remark}
\newtheorem{rem}[thm]{\protect\remarkname}
\theoremstyle{plain}
\newtheorem{lem}[thm]{\protect\lemmaname}
\theoremstyle{definition}
\newtheorem{defn}[thm]{\protect\definitionname}
\newlist{thmstepnv}{enumerate}{4}
\setlist[thmstepnv]{leftmargin=*,align=left,wide,labelwidth=0pt,labelindent=0pt}
\setlist[thmstepnv,1]{label={\itshape {\thmstepname} \arabic*.},ref=\arabic*}
\setlist[thmstepnv,2]{label={\itshape {\thmstepname} {\thethmstepnvi\alph*}.},ref=\thethmstepnvi\alph*}
\setlist[thmstepnv,3]{label={\itshape {\thmstepname\ \alph*.}},ref=\alph*}
\setlist[thmstepnv,4]{label={\itshape {\thmstepname} \arabic*.},ref=\arabic*}
\newcommand{\cref}[1]{\text{\zcref{#1}}}

\hypersetup{final}
\usepackage{mleftright}
\mleftright

\providecommand{\corollaryname}{Corollary}
\providecommand{\definitionname}{Definition}
\providecommand{\lemmaname}{Lemma}
\providecommand{\propositionname}{Proposition}
\providecommand{\remarkname}{Remark}
\providecommand{\theoremname}{Theorem}
\providecommand{\thmstepname}{Step}

\usepackage[style=trad-alpha,isbn=false,doi=false,url=false,maxnames=99,maxalphanames=99]{biblatex}
\DeclareLabelalphaNameTemplate{
  \namepart[use=true, pre=true, strwidth=1, compound=true]{prefix}
  \namepart[compound=true]{family}
}
\AtEveryBibitem{\clearfield{month}}
\AtEveryBibitem{\clearfield{primaryclass}}
\AtEveryBibitem{\clearfield{eprintclass}}
\DeclareSourcemap{
  \maps[datatype=bibtex]{
    \map{
      \pertype{book}
      \step[fieldsource=pages, final]
      \step[fieldset=pagetotal, origfieldval, final]
      \step[fieldset=pages, null]
    }
  }
}
\global\long\def\supp{\operatorname{supp}}%

\global\long\def\Uniform{\operatorname{Uniform}}%

\global\long\def\dif{\mathrm{d}}%

\global\long\def\e{\mathrm{e}}%

\global\long\def\Var{\operatorname{Var}}%

\global\long\def\id{\operatorname{id}}%

\global\long\def\e{\mathrm{e}}%

\global\long\def\p{\mathrm{p}}%

\global\long\def\Law{\operatorname{Law}}%

\global\long\def\supp{\operatorname{supp}}%

\global\long\def\dif{\mathrm{d}}%

\global\long\def\eps{\varepsilon}%

\def\viiva{\hspace{.7pt}\vert\hspace{.8pt}}

\global\long\def\EE{\mathbb{E}}%

\global\long\def\RR{\mathbf{R}}%

\global\long\def\ZZ{\mathbb{Z}}%

\global\long\def\NN{\mathbb{N}}%
\RequirePackage{color}
\definecolor{darkgreen}{rgb}{0.0,0.5,0.0}
\definecolor{indigo}{rgb}{0.3,0,0.5}

\global\long\def\st{\mathrel{:}}%
\begin{document}

\title{Viscous shock fluctuations in KPZ}
\author{Alexander Dunlap\thanks{Department of Mathematics, Duke University, Durham, NC 27708, USA. Email: \href{mailto:dunlap@math.duke.edu}{\url{dunlap@math.duke.edu}}.}\and Evan
Sorensen\thanks{Department of Mathematics, Columbia University, New York, NY 10027, USA. Email: \href{mailto:evan.sorensen@columbia.edu}{\url{evan.sorensen@columbia.edu}}.}}
\maketitle
\begin{abstract}
We study ``V-shaped'' solutions to the KPZ equation, those having
opposite asymptotic slopes $\theta$ and $-\theta$, with $\theta>0$, at positive and negative infinity, respectively. 
Answering
a question of Janjigian, Rassoul-Agha, and Sepp\"al\"ainen, we show that
the spatial increments of V-shaped solutions cannot be statistically
stationary in time. This completes the classification of statistically
time-stationary spatial increments for the KPZ equation by ruling out the last
case left by those authors.

To show that these V-shaped time-stationary measures do not exist, we study the location of the corresponding  ``viscous shock,'' which, roughly speaking, is the location of the bottom of the V. We describe the limiting rescaled fluctuations, and in particular show that the fluctuations of the shock location are not tight, for both stationary and flat initial data.
We also show that if the KPZ equation is started with V-shaped initial data, then the  long-time limits of the time-averaged laws of the spatial increments of the solution are mixtures of the laws of the spatial increments of $x\mapsto B(x)+\theta x$ and $x\mapsto B(x)-\theta x$, where $B$ is a standard two-sided Brownian motion. 
\end{abstract}
\tableofcontents{}

\section{Introduction}

We consider the KPZ equation formally given by
\begin{equation}
\dif h(t,x)=\frac{1}{2}[\Delta h(t,x)+(\partial_{x}h(t,x))^{2}]\dif t+\dif W(t,x),\label{eq:KPZ}
\end{equation}
where $\dif W$ is a space-time white noise. Classically, this equation is ill-posed; indeed, one must go through a limiting argument and subtract off an infinite renormalization term on the right-hand side to make proper sense of the equation. As is standard in the study of the KPZ equation, we avoid this issue by considering the Cole--Hopf (physical) solutions to \cref{eq:KPZ}.
These are given by $h=\log\phi$, where $\phi$ solves the stochastic
heat equation
\begin{equation}
\dif\phi(t,x)=\frac{1}{2}\Delta\phi(t,x)\dif t+\phi(t,x)\dif W(t,x).\label{eq:SHE}
\end{equation}

The long-time behavior of solutions to \cref{eq:KPZ} has been
the subject of significant study in the past several decades. It is
now known that the KPZ equation is a member of the KPZ universality
class \cite{Sasamoto-Spohn-2010b,Sasamoto-Spohn-2010c,Sasamoto-Spohn-2010a,Calabrese-LeDoussal-Rosso-2010,Dotsenko-2012,Balazs-Quastel-Seppalainen-2011,Amir-Corwin-Quastel-2011,BCFV15,KPZ_equation_convergence,heat_and_landscape,Wu-23}, and in
particular that it exhibits nontrivial fluctuations under the ``$1:2:3$
scaling.'' 
In other
words, the rescaled function $\eps h(\eps^{-3}t,\eps^{-2}x)$ converges to
a nontrivial limit, called the KPZ fixed point \cite{KPZfixed}, as
$\eps\to0$. Implicit in this scaling is that the fluctuations of the solutions to \cref{eq:KPZ} grow as $t\to\infty$, and in particular there
are no invariant measures for this equation.

On the other hand, the \emph{recentered} process $h(t,x)-h(t,0)$
is known to have $O(1)$ fluctuations as $t\to\infty$, and indeed
to admit invariant measures. For $\theta\in\RR$, if we let $\mu_{\theta}$
denote the law of standard two-sided Brownian motion with drift $\theta$,
then $\mu_{\theta}$ is invariant under the dynamics of $h(t,x)-h(t,0)$
\cite{Bertini-Giacomin-1997,Funaki-Quastel-2015,Janj-Rass-Sepp-22,Gu-Quastel-24}.
We note that if $\tilde{f}\sim\mu_{\theta}$, then, according to a standard
property of Brownian motion with drift, we have
\[
\lim_{x\to\pm\infty}\frac{\tilde{f}(x)}{x}=\theta.
\]

It has been conjectured that this one-parameter family $(\mu_{\theta})_{\theta\in\RR}$
in fact comprises \emph{all} invariant measures for the recentered
process; see e.g. \cite[Remark 1.1]{Funaki-Quastel-2015}. Great
progress on this question has been made in \cite{Janj-Rass-Sepp-22},
in which it was shown that if $\mu$ is an extremal invariant measure
for the recentered process, then either there is some $\theta\in\RR$
such that $\mu=\mu_{\theta}$, or else if $\tilde{f}\sim\mu$, there exists $\theta > 0$ so that

\begin{equation}
\lim_{x\to\pm\infty}\frac{\tilde{f}(x)}{|x|}=\theta\qquad\text{a.s.}\label{eq:Vshaped}
\end{equation}
We call functions satisfying \cref{eq:Vshaped} ``V-shaped''
since they asymptotically look like the shape of a capital letter
``V.'' The condition $\theta > 0$ is significant. Indeed, \cite[Theorem 3.23]{Janj-Rass-Sepp-22} shows that, if started from an initial condition satisfying \eqref{eq:Vshaped} a.s. for $\theta \le 0$, then the recentered solution to the KPZ equation converges to Brownian motion with zero drift. (Such initial conditions correspond to rarefaction fans.) In particular, it is already known that there are no extremal invariant measures $\mu$ supported on functions satisfying \eqref{eq:Vshaped} for $\theta < 0$.

In \cite[Open Problem 1]{Janj-Rass-Sepp-22}, the authors asked whether
invariant measures supported on functions satisfying \cref{eq:Vshaped} for $\theta > 0$ actually exist.
One of the main results of the present work is that the answer is no. In the
following theorem statement, $\mathcal{C}_{\mathrm{KPZ};0}$ is the
natural function space for the recentered KPZ equation; see \cref{eq:CKPZ0def}
below and the discussion in \cite[Section~2.3]{Janj-Rass-Sepp-22}.
\begin{thm}
\label{thm:no-v-shaped}For $\theta > 0$, there does not exist an invariant measure
 for the recentered process of the
KPZ equation on $\mathcal{C}_{\mathrm{KPZ};0}$ that is supported on functions satisfying \cref{eq:Vshaped}.
\end{thm}

As a corollary of this and \cite[Theorem 3.26(ii)]{Janj-Rass-Sepp-22},
we obtain the complete characterization of extremal invariant measures
for the recentered KPZ equation. We note that, in \cite{Janj-Rass-Sepp-22}, they study invariant measures on a slightly different space; that is, the space of equivalence classes of functions, where two functions are equivalent if their difference is a constant. Since our choice of the space $\mathcal{C}_{\mathrm{KPZ};0}$ pins the functions at $0$, the two notions are equivalent. 
\begin{cor}
\label{cor:classification}If $\mu$ is an extremal invariant measure
on $\mathcal{C}_{\mathrm{KPZ};0}$ for the recentered KPZ equation,
then there is some $\theta\in\RR$ such that $\mu=\mu_{\theta}$.
\end{cor}

Previous results on asymmetric simple exclusion processes (ASEP) have obtained
complete understandings of the invariant measures; see \cref{subsec:ASEP-comparison}
below. We emphasize, however, that while properly-rescaled ASEP converges
to the KPZ equation, the characterization of invariant measures in
ASEP does not immediately pass to the limit. Indeed, one must rule
out invariant measures that do not arise as scaling limits of invariant
measures for ASEP. This question of characterizing the stationary
measures in the context of the KPZ equation had previously been conjectured
and discussed in several works \cite{Funaki-Quastel-2015,Janj-Rass-Sepp-22,Krug-1992,Spohn-14}
before its final resolution here.

The first author and Ryzhik studied V-shaped solutions to the KPZ
equation, but with $\dif W$ replaced by a noise that is spatially
smooth and white in time, in \cite{Dunlap-Ryzhik-2020}. In fact,
that paper worked with the \emph{gradient} of the KPZ equation, the
stochastic Burgers equation. This is equivalent to the setting we
have been considering since studying the gradient is equivalent to
subtracting the value at $0$. The starting point of the analysis
in \cite{Dunlap-Ryzhik-2020} was the observation (at the level of
the stochastic Burgers equation) that V-shaped solutions to \cref{eq:KPZ}
can be constructed from two solutions to \cref{eq:KPZ}, with
stationary spatial increments, that are driven by the same noise.
(In the smooth-noise setting, the ergodic behavior of such solutions
was studied in \cite{Dunlap-Graham-Ryzhik-21}.) Specifically, if
$h_{+}$ and $h_{-}$ are two solutions to \cref{eq:KPZ} driven
by the same noise, then
\begin{equation}
h_{\mathsf{V}}(t,x)=V[h(t,\cdot)](x)\coloneqq\log\frac{\e^{h_{+}(t,x)}+\e^{h_{-}(t,x)}}{2}\label{eq:Vshapedsolution}
\end{equation}
is also a solution to \cref{eq:KPZ}. If $\theta>0$ and $\lim\limits_{|x|\to\infty}\frac{h_{\pm}(t,x)}{x}=\pm\theta$,
then it is clear from \cref{eq:Vshapedsolution} that $\lim\limits_{|x|\to\infty}\frac{h_{\mathsf V}(t,x)}{|x|}=\theta$,
which means that $h$ is a V-shaped solution.

To study potential V-shaped \emph{stationary} solutions, we write
a centered version of \cref{eq:Vshapedsolution} as
\begin{align}
h_{\mathsf{V}}(t,x)-h_{\mathsf{V}}(t,0) & =\log\frac{\e^{h_{+}(t,x)}+\e^{h_{-}(t,x)}}{\e^{h_{+}(t,0)}+\e^{h_{-}(t,0)}}\nonumber \\
 & =\log\frac{\e^{h_{+}(t,x)-h_{+}(t,0)}+\e^{h_{-}(t,x)-h_{-}(t,0)-(h_{+}(t,0)-h_{-}(t,0))}}{\e^{h_{+}(t,0)-h_{+}(t,0)}+\e^{-(h_{+}(t,0)-h_{-}(t,0))}}.\label{eq:hvintermsofstationaryquantities}
\end{align}
This formula depends only on $h_{+}(t,x)-h_{+}(t,0)$, $h_{-}(t,x)-h_{-}(t,0)$,
and $h_{+}(t,0)-h_{-}(t,0)$. The first two quantities have stationary
versions, but we will see that the last one in fact grows in time
and does not have a stationary distribution. Informally speaking,
this is the ``reason'' for the lack of V-shaped stationary solutions
claimed in \cref{thm:no-v-shaped}.

The jointly stationary solutions of the process $\left(h_{-}(t,x)-h_{-}(t,0),h_{+}(t,x)-h_{+}(t,0)\right)$,  were recently described in the work \cite{GRASS-23}
by Groathouse, Rassoul-Agha, Sepp\"al\"ainen, and the second author. More generally, there is an explicit description of the law of the jointly invariant measures
for the recentered solutions of \cref{eq:KPZ}, with $k$ solutions driven by the same noise
(but with different asymptotic slopes) for any $k\in\NN$.
We restrict our present discussion to $k=2$, as that is what we will
use in the present paper. We also restrict to the case of opposite drifts, noting there is also a description for general choice of drifts. Let $B_{1},B_{2}$ be two independent standard two-sided
Brownian motions (with $B_{1}(0)=B_{2}(0)=0$), and define
\begin{align}
f_{-}(x) & =B_{1}(x)-\theta x,\label{eq:f-}\\
f_{+}(x) & =B_{2}(x)+\theta x+\mathcal{S}_{\theta}(x)-\mathcal{S}_{\theta}(0),\label{eq:f+}
\end{align}
where
\begin{equation}
\mathcal{S}_{\theta}(x)=\log\int_{-\infty}^{x}\exp\left\{ (B_{2}(y)-B_{2}(x))-(B_{1}(y)-B_{1}(x))+2\theta(y-x)\right\} \,\dif y.\label{eq:Stheta}
\end{equation}
We define 
\begin{equation}
\nu_{\theta}=\Law((f_{-},f_{+})).\label{eq:nuthetadef}
\end{equation}
 It is shown in \cite[Theorem 1.1]{GRASS-23} that if $h_{-}$ and $h_{+}$ are
two solutions to \cref{eq:KPZ} with initial data $(h_{-},h_{+})(0,x)\sim\nu_{\theta}$
independent of the noise, then $(h_{-},h_{+})(t,x)-(h_{-},h_{+})(t,0)\sim\nu_{\theta}$
as well. Also, we have $\Law(h_{\pm}(t,\cdot))=\mu_{\pm\theta}$; i.e.,
the marginals of $\nu_{\theta}$ are the laws of two-sided Brownian motions with opposite
drifts. See \cite{OCY2001,MY05} and their references for earlier studies of the integrals of the exponentials of Brownian motion such as those appearing in \cref{eq:Stheta}.

Key to the proof of \cref{thm:no-v-shaped} will be the following
theorem on the fluctuations of $(h_{+}-h_{-})(t,0)$:
\begin{thm}
\label{thm:zeroval-stationarity}Let $\theta>0$, and let $h_{+}$
and $h_{-}$ solve \cref{eq:KPZ} with initial data $(h_{-},h_{+})(0,x)\sim\nu_{\theta}$
independent of the noise. Then we have the convergence in distribution
\begin{equation}
\frac{h_{+}(t,0)-h_{-}(t,0)}{t^{1/2}}\Longrightarrow\mathcal{N}(0,2\theta)\label{eq:zeroval-stationary}
\end{equation}
as $t\to\infty$.
\end{thm}

We emphasize that \cref{thm:zeroval-stationarity} is sensitive to the choice of initial data, even at the level of the scaling exponent. Indeed, by contrast, we have the following analogous
result for flat initial data.
\begin{thm}
\label{thm:zeroval-flat}Let $\theta>0$, and let $h_{+}$ and $h_{-}$
solve \cref{eq:KPZ} with initial data $h_{\pm}(0,x)=\pm\theta x$.
Let $X_{1}$ and $X_{2}$ denote two independent Tracy--Widom GOE
random variables. Then we have the convergence in distribution
\begin{equation}
\frac{h_{+}(t,0)-h_{-}(t,0)}{t^{1/3}}\Longrightarrow\frac{X_{1}-X_{2}}{2}\label{eq:GOEconv}
\end{equation}
as $t\to\infty$.
\end{thm}
The limiting objects obtained in \cref{thm:zeroval-stationarity,thm:zeroval-flat} have previously been obtained in \cite{Ferrari-Fontes-1994b} and \cite{Ferrari-Ghosal-Nejjar-2019}, respectively, as limits of certain roughly-analogous quantities related to ASEP; see \cref{subsec:ASEP-comparison} for a discussion. There, we also discuss the method of proof and contrast from the methods used for ASEP.

\begin{rem} \label{rem:shear_dif}
One may ask about the joint solutions to the KPZ equation with asymptotic drifts that are not opposite. Indeed, \cite{GRASS-23} studies more general measures $\nu_{\theta_1,\theta_2}$, which are jointly invariant and have marginals of Brownian motions with drift $\theta_1 < \theta_2$. By \cite[Theorem 2.11(ii)]{GRASS-23}, if $(f_-,f_+) \sim \nu_{\theta}$ with $\theta = \frac{\theta_2 - \theta_1}{2}$, then $(f_1(x) + \frac{\theta_1 + \theta_2}{2}x,f_2(x) + \frac{\theta_1 + \theta_2}{2}x) \sim \nu_{\theta_1,\theta_2}$. Using this fact and the shear invariance in \cref{eq:Z_shear}, if $(h_-,h_+)(0,x) \sim \nu_{\theta_1,\theta_2}$, then we have the convergence in distribution  
\[
\frac{h_+\left(t, -\frac{\theta_1 + \theta_2}{2}t\right) - h_-\left(t, -\frac{\theta_1 + \theta_2}{2}t\right)}{t^{1/2}}\Longrightarrow \mathcal{N}(0,\theta_2 - \theta_1). 
\]
Also, if we start from the initial condition $h_{-}(0,x) = \theta_1 x$ and $h_+(0,x) = \theta_2 x$, then
\[
\frac{h_+\left(t, -\frac{\theta_1 + \theta_2}{2}t\right) - h_-\left(t, -\frac{\theta_1 + \theta_2}{2}t\right)}{t^{1/3}} \Longrightarrow \frac{X_1 - X_2}{2}.
\]
In this more general setting, the term $-\frac{\theta_1 + \theta_2}{2}$ represents the asymptotic velocity of the shock (which is zero if $\theta_1 = -\theta_2$); see also \cref{rem:shear_shock}.
\end{rem}

\subsection{Long-time behavior of V-shaped solutions}

Given that \cref{thm:no-v-shaped} tells us that there are no
spacetime-stationary V-shaped solutions, it is natural to ask about the behavior of solutions that are started with V-shaped initial
data. The following theorem says that if a solution to the KPZ equation
starts with V-shaped initial data with slopes $-\theta$ and $\theta$
at $-\infty$ and $+\infty$, then the laws of its recentered versions
are tight, and any subsequential limits must be mixtures of $\mu_{-\theta}$
and $\mu_{\theta}$. In the statement, $\mathcal{C}_{\mathrm{KPZ}}$ is the natural function space for the KPZ equation without recentering; see \cref{eq:CKPZdef} below.
\begin{thm}
\label{thm:whatcanyouconvergeto}Let $\theta>0$ and suppose that
$h_{\mathsf{V}}$ is a solution to \cref{eq:KPZ} with initial
condition $h_{\mathsf{V}}(0,\cdot)\in\mathcal{C}_{\mathrm{KPZ}}$
satisfying
\begin{equation}
\lim_{|x|\to\infty}\frac{h_{\mathsf{V}}(0,x)}{|x|}=\theta.\label{eq:Vshaped-1}
\end{equation}
Then the following properties hold:
\begin{enumerate}
\item \label{enu:Vtight}The family of random variables $(h_{\mathsf{V}}(t,\cdot)-h_{\mathsf{V}}(t,0))_{t\ge 0}$
is tight with respect to the topology of $\mathcal{C}_{\mathrm{KPZ};0}$.
\item \label{enu:convtomixture}Let $U_{T}\sim\Uniform([0,T])$ be independent of everything else. If $m$ is a probability measure on $\mathcal{C}_{\mathrm{KPZ};0}$
and $T_{k}\uparrow\infty$ is a sequence such that
\begin{equation}
\Law(h_{\mathsf{V}}(U_{T_{k}},\cdot)-h_{\mathsf{V}}(U_{T_{k}},0))\to m\label{eq:lawsconvtom}
\end{equation}
weakly as $k\to\infty$, then there exists a $\zeta\in[0,1]$ (possibly
depending on the choice of subsequence) such that $m=(1-\zeta)\mu_{-\theta}+\zeta\mu_{\theta}$.
\end{enumerate}
\end{thm}

Basins of attraction of the invariant measures of the KPZ equation
have been a topic of great interest in the literature. Extensive results
were obtained in \cite{Janj-Rass-Sepp-22}, where it was shown that,
for $\theta>0$, if an initial condition satisfies
\[
\lim_{x\to+\infty}\frac{h(0,x)}{x}=\theta\qquad\text{and}\qquad\lim_{x\to-\infty}\frac{h(0,x)}{x}>-\theta,
\]
then $x\mapsto h(t,x)-h(t,0)$ conveges in distribution to a two-sided
Brownian motion with drift $\theta$. There are similar descriptions
of the basin of attraction in cases when $\theta=0$ and $\theta<0$.
This is analogous to the ergodic theorems of Liggett \cite{Liggett1975}
for ASEP, where descriptions of the basin of attraction are described
depending on the asymptotic density of particles to the left and right
of the origin. Descriptions of basins of attraction have been obtained
for the Burgers equation with various types of non-integrable forcing
in \cite{Bakhtin-Cator-Konstantin-2014,Bakhtin-16,Bakhtin-Li-18,Bakhtin-Li-19,Dunlap-Graham-Ryzhik-21},
and for the KPZ fixed point in \cite{Busa-Sepp-Sore-22a}. However,
left open in all of these works is the limiting behavior of the increment
process when started from an initial condition satisfying \cref{eq:Vshaped},
which is what is considered in \cref{thm:whatcanyouconvergeto}. 

One may ask, in relation to \cref{thm:whatcanyouconvergeto}, whether a stronger result is possible. That is, does there exist a universal  value of $\zeta$ for all subsequential limits. Without further assumptions on the rate of convergence to the slopes $\pm \theta$ at $\pm \infty$, one does not expect to obtain such a statement. In the setting of ASEP, Liggett \cite{Liggett1975} demonstrated the existence of initial V-shaped configurations such that the analogues of the extremal measures $\mu_{-\theta}$ and $\mu_{\theta}$ are both seen as subsequential limits. In the setting of ASEP, one considers configurations $\eta \in \{0,1\}^\ZZ$ of particles and holes having asymptotic densities $\lambda$ and $\rho$ to the left and right of the origin, respectively, and the case $\lambda + \rho = 1$ is the analogue of the V-shaped solution. On the other hand, it was conjectured in \cite{Liggett1975} and proved in \cite{Andjel-Bramson-Liggett-1988} that, for $\lambda+ \rho = 1$, if $\mu$ is a product measure on the space of configurations, and the restrictions of $\mu$ to the left and right of the origin are close enough to i.i.d. Bernoulli measures, in the sense that 
\[
\sum_{x = - \infty}^0 |\mu(\eta: \eta(x) = 1) - \lambda| + \sum_{x = 0}^\infty |\mu(\eta:\eta(x) = 1) - \rho| < \infty,
\]
then the process converges in law to the symmetric mixture of the two i.i.d. Bernoulli measures with intensities $\lambda$ and $\rho$. This condition can be thought of as an approximate symmetry between the configurations on the left and right, leading to a symmetric mixture. Of course, in the setting of \cref{thm:whatcanyouconvergeto}, if the initial V-shaped data $h_{\mathsf{V}}(0,\cdot)$ satisfies $h_{\mathsf{V}}(0,\cdot)\overset{\mathrm{law}}=h_{\mathsf{V}}(0,-\cdot)$, then the symmetry must pass to the limit, and $\zeta=1/2$. In particular, for the two cases considered in this paper, namely $(h_-(0,x),h_+(0,x)) \sim \nu_\theta$ and $(h_-(0,x),h_+(0,x)) = (-\theta x,\theta x)$,  we have $\zeta = \frac{1}{2}$ for all subsequential limits. We leave the precise study of the dependence of the possible subsequential limits on the initial data to future work.

We can also study the behavior of V-shaped solutions started at a large negative time and considered at time $0$. In this case, we can study almost-sure limiting behavior, rather than behavior in law. For each $\theta\in\RR$, it was shown in \cite{Janj-Rass-Sepp-22} that, there is a random process $\overline{\mathbf{f}}=(\overline f_-,\overline f_+)$ (on the same probability space as the noise) such that, if $\mathbf h^T=(h_-^T,h_+^T)$ is a vector of solutions to \cref{eq:KPZ} with initial condition $h^T_\pm(-T,\cdot)$ in the basin of attraction for $\nu_\theta$, then $\lim\limits_{T\to\infty} [\mathbf h^T(0,\cdot)-\mathbf h^T(0,0)] = \overline{\mathbf{f}}$ almost surely. (See \cref{prop:spatially-homog-convergence} below for the precise statement we will use.) We can use this theorem to prove the following, which is a partial solution to \cite[Open Problem~6]{Janj-Rass-Sepp-22}.
\begin{thm}\label{thm:asconvergence}
    There exists an event of probability one on which the following holds. Let  $f_{\mathsf{V}}$ be a continuous function satisfying $\lim\limits_{|x| \to \infty} \frac{f_{\mathsf{V}}(x)}{|x|} = \theta$. Let $h_{\mathsf{V}}^T$ be a solution to \cref{eq:KPZ} with initial condition $h_{\mathsf{V}}^T(-T,x)=f_{\mathsf{V}}(x)$.   For any sequence $T_k\uparrow\infty$, there exists a (possibly random)  subsequence $T_{k_\ell}\uparrow\infty$ and a $\xi\in[0,1]$ such that
    \begin{equation}
        \lim_{\ell\to\infty} [h^{T_{k_\ell}}(0,\cdot)-h^{T_{k_\ell}}(0,0)] = \log(\xi \e^{\overline f_-}+(1-\xi)\e^{\overline f_+})\label{eq:asconvergence}
\end{equation}
in the topology of $\mathcal{C}_{\mathrm{KPZ};0}$.
\end{thm}
In this theorem, we expect that in general $\xi$ will depend on the choice of subsequence. Indeed, we expect there will be subsequences with $\xi \not\in \{0,1\}$, even though \cref{thm:whatcanyouconvergeto} suggests that $\xi$ should be either $0$ or $1$ for ``typical'' sequences. This is because, even though we expect the shock location to typically be large, it may oscillate from large negative to large positive, and hence there may be infinite sequences of times for which it is of order $1$. We expect that the methods of this paper could be used to prove stronger statements in this direction in the case of particularly symmetric initial data, such as when the initial data is taken to be $\theta |x|$ for some $\theta>0$, or when it is taken to be an a random ``almost-stationary''' V-shaped solution as considered below.

\subsection{The reference frame of the shock}

While the paper \cite{Dunlap-Ryzhik-2020} does not consider the existence
of V-shaped stationary solutions in the smooth-noise setting, it does show that
there are invariant measures for V-shaped solutions if they are recentered
not just vertically but also horizontally. More precisely, in that
paper, it was shown that for solutions $h_{+}$
and $h_{-}$ to \cref{eq:KPZ} with different asymptotic slopes,
if we define $h_{\mathsf{V}}$ by \cref{eq:Vshapedsolution},
then there is a process $(b_{t})_{t\ge0}$ such that the process 
\[
x\mapsto(h_{-},h_{+},h_{\mathsf{V}})(t,b_{t}+x)-(h_{-},h_{+},h_{\mathsf{V}})(t,b_{t}).
\]
admits an invariant measure. In other words, the \emph{shape} of the
V-shaped solution is preserved in time, even if the \emph{location}
of the center of the V moves as time advances. The shock location
$b_{t}$ interacts with the local geometry of $h_{-}$ and $h_{+}$,
so the projection of this invariant measure onto the first two coordinates
is not the same as $\nu_{\theta}$. The following theorem gives this
tilt and the precise statement of the stationarity in the space-time white noise case. It is the spacetime-white-noise analogue of \cite[Theorem 1.1]{Dunlap-Ryzhik-2020}. It is also analogous to the result \cite[Theorem 2.3]{Ferrari-Kipnis-Saada-1991} for ASEP. There, the description of the stationary measure is much more complicated; it is constructed as an average of empirical measures seen from a second-class particle. 
\begin{thm}
\label{thm:nuhat}We define the measure $\hat{\nu}_{\theta}$ that
is absolutely continuous with respect to $\nu_{\theta}$ with Radon--Nikodym
derivative
\begin{equation}
\frac{\dif\hat{\nu}_{\theta}}{\dif\nu_{\theta}}(f_{-},f_{+})=\frac{1}{2\theta}\partial_{x}(f_{+}-f_{-})(0).\label{eq:RNderiv}
\end{equation}
Let $(h_{-},h_{+},h_{\mathsf{V}})$ be a vector of solutions to \cref{eq:KPZ}
with initial condition $(h_{-},h_{+})(0,\cdot)\sim\hat{\nu}_{\theta}$
and $h_{\mathsf{V}}(0,\cdot)=V[(h_{-},h_{+})(0,\cdot)]$. Then the
following statements hold.
\begin{enumerate}
\item There is a random process $(b_{t})_{t\ge0}$ such that, for each $t\ge0$,
$b_{t}$ is the unique $x\in\RR$ such that 
\begin{equation}
h_{-}(t,x)=h_{+}(t,x).\label{eq:hplusishminus}
\end{equation}
\item For each $t\ge0$, we have
\[
\Law\left((h_{-},h_{+},h_{\mathsf{V}})(t,b_{t}+\cdot)-(h_{-},h_{+},h_{\mathsf{V}})(t,\cdot)\right)=\hat{\nu}_{\theta}.
\]
\end{enumerate}
\end{thm}

The proof of \cref{thm:nuhat} follows that of \cite[Theorem 1.1]{Dunlap-Ryzhik-2020}.
The only technical point in this case is that, because $(f_{-},f_{+})\sim\nu_{\theta}$
are not differentiable processes, one may ask whether the Radon--Nikodym
derivative \cref{eq:RNderiv} is well-defined. This is in fact
not an issue since the difference $f_{+}-f_{-}$ is differentiable
almost surely, even though $f_{-}$ and $f_{+}$ are individually
not differentiable. This can be seen from the formulas \zcref[range]{eq:f-,eq:Stheta}:
we can write
\begin{align}
 (f_{+}-f_{-})(x)  &=B_{2}(x)-B_{1}(x)+2\theta x+\mathcal{S}_{\theta}(x)-\mathcal{S}_{\theta}(0)\nonumber \\
  &=B_{2}(x)-B_{1}(x)+2\theta x\nonumber\\&\qquad+\log\frac{\int_{-\infty}^{x}\exp\left\{ (B_{2}(y)-B_{2}(x))-(B_{1}(y)-B_{1}(x))+2\theta(y-x)\right\} \,\dif y}{\int_{-\infty}^{0}\exp\left\{ B_{2}(y)-B_{1}(y)+2\theta y\right\} \,\dif y}\nonumber \\
 &=\log\frac{\int_{-\infty}^{x}\exp\left\{ B_{2}(y)-B_{1}(y)+2\theta y\right\} \,\dif y}{\int_{-\infty}^{0}\exp\left\{ B_{2}(y)-B_{1}(y)+2\theta y\right\} \,\dif y},\label{eq:fplusfminus}
\end{align}
which is evidently differentiable in $x$. Indeed, the derivative is given by
\begin{equation}
\partial_{x}(f_{+}-f_{-})(x)=\left(\int_{-\infty}^{x}\exp\left\{ (B_{2}-B_1)(y)-(B_2-B_1)(x)+2\theta (y-x)\right\} \,\dif y\right)^{-1}.\label{eq:deriv-intro}
\end{equation}
As expected, this expression is statistically stationary in $x$. In facts, it is known to be a Gamma-distributed random variable; see \cref{eq:Rdef}ff.\ below.

The work \cite{Dunlap-Ryzhik-2020} did not address the fluctuations
of $b_{t}$. In the present setting, by contrast, we are able to do
this. In fact, because $h_{-}$ and $h_{+}$ both look linear on large
scales, the fluctuations of $b_{t}$ are closely related to the fluctuations
of $h_{+}(t,0)-h_{-}(t,0)$ discussed in \cref{thm:zeroval-stationarity,thm:zeroval-flat}. We state the following theorem
on the location of the shock for both the stationary and flat initial
conditions covered in those two theorems, as well as the shock-reference-frame-stationary
initial condition discussed in \cref{thm:nuhat}.
\begin{thm}
\label{thm:btflucts}Let $\theta>0$.
\begin{enumerate}
\item \label{enu:stationaryic}Let $h_{+}$ and $h_{-}$ solve \cref{eq:KPZ}
with initial data $(h_{-},h_{+})(0,x)\sim\nu_{\theta}$ independent
of the noise. For each $t\ge0$, there is a unique $b_{t}\in\RR$
such that \cref{eq:hplusishminus} holds, and we have the convergence
in distribution
\begin{equation}
t^{-1/2}b_{t}\Longrightarrow\mathcal{N}(0,(2\theta)^{-1}).\label{eq:stationary-bt}
\end{equation}
\item \label{enu:flatic}Let $h_{+}$ and $h_{-}$ solve \cref{eq:KPZ}
with initial data $h_{\pm}(0,x)=\pm\theta x$. For each $t\ge0$,
there is a unique $b_{t}\in\RR$ such that \cref{eq:hplusishminus}
holds, and we have the convergence in distribution
\begin{equation}
t^{-1/3}b_{t}\Longrightarrow\frac{1}{4\theta}(X_{1}-X_{2}),\label{eq:flat-bt}
\end{equation}
where $X_{1}$ and $X_{2}$ are independent Tracy--Widom GOE random
variables.
\item \label{enu:tiltedic}Let $h_{+}$ and $h_{-}$ solve \cref{eq:KPZ}
with initial data $(h_{-},h_{+})(0,x)\sim\hat{\nu}_{\theta}$ independent
of the noise. For each $t\ge0$, there is a unique $b_{t}\in\RR$
such that \cref{eq:hplusishminus} holds. We have the convergence
in distribution 
\begin{equation}
t^{-1/2}[h_{+}(t,0)-h_{-}(t,0)]\Longrightarrow\mathcal{N}(0,2\theta)\label{eq:shockreferenceframe-zerodiff}
\end{equation}
and
\begin{equation}
t^{-1/2}b_{t}\Longrightarrow\mathcal{N}(0,(2\theta)^{-1}).\label{eq:shockreferenceframe-bt}
\end{equation}
\end{enumerate}
\end{thm}

\begin{rem} \label{rem:shear_shock}
    As in \cref{rem:shear_dif}, the shear invariance of the KPZ equation allows us to immediately derive the asymptotics of the shock when started from initial conditions with non-opposite slopes. For $\theta_1 < \theta_2$, if $(h_-,h_+)(0,x) \sim \nu_{\theta_1,\theta_2}$ or $\hat \nu_{\theta_1,\theta_2}$ (defined analogously as in \cref{thm:nuhat}), we have 
    \[
    t^{-1/2}\Bigl(b_t  + \frac{\theta_1 + \theta_2}{2}t\Bigr) \Longrightarrow \mathcal N(0, (\theta_2 - \theta_1)^{-1}).
    \]
    For $h_-(0,x) = \theta_1 x$ and $h_+(0,x) = \theta_2 x$, we have 
    \[
    t^{-1/3}\Bigl(b_t + \frac{\theta_1 + \theta_2}{2}t\Bigr) \Longrightarrow \frac{1}{2(\theta_2 - \theta_1)}(X_1 - X_2).
    \]
    From these expressions we see that $-(\theta_1+\theta_2)/2$ is the asymptotic velocity of $b_t$.
\end{rem}

\begin{rem}
    In the cases where $(h_-,h_+)(0,x) \sim \nu_\theta$ or $\hat\nu_{\theta}$, the proof suggests that the full time-scaling limit of $b_t$ should be a Brownian motion with drift $1/2\theta$. 
\end{rem}

\subsection{Comparison with previous work on ASEP}\label{subsec:ASEP-comparison}

Given a Markov process, it is natural to try to characterize all of
its extremal (time-ergodic) invariant measures. This question has
been studied in depth in the context of the simple exclusion process
first introduced by Spitzer \cite{Spitzer-1970}. Early works by Spitzer
and Liggett provided proofs that i.i.d.~Bernoulli measures are the
only extremal stationary measures for the simple exclusion process
in the case when the transition rates are symmetric in space \cite{Liggett-1973,Liggett1974a,Spitzer1974},
and in the case when the Markov chain is positive recurrent and reversible
\cite{Liggett1974b}. The symmetries assumed in those settings substantially
simplified the problem. Another case that is particularly relevant
to the present work is that of the asymmetric simple exclusion process
(ASEP) on $\ZZ$, where Liggett showed in \cite{Liggett1976}
that the only extremal stationary measures are the i.i.d.~Bernoulli
measures and a family of measures that are supported on configurations
with only finitely many holes on the line (known as blocking measures). The ASEP case is particularly
relevant because the model is known to converge to the KPZ equation under
the weak asymmetry scaling \cite{Bertini-Giacomin-1997} (see also \cite{Parekh-2023}). Under this
scaling limit, one centers around a fixed characteristic direction,
and the height functions of the i.i.d.~Bernoulli measures converge
to Brownian motion with drift, while the height functions for the
other invariant measures explode.

The methods of proof in the present paper are quite different from
those for ASEP. Indeed, the work of \cite{Liggett1976} makes heavy
use of local and discrete arguments. However, there are similarities in the broad approach, in
the sense that we use couplings of invariant measures that are jointly
invariant for the process. In the particle systems context, the natural joint
evolution is known as the basic coupling \cite{Liggett1974b,Liggett1975,Spitzer1974}.
The proof in \cite{Liggett1976} heuristically proceeds by showing
that, when comparing any two invariant measures $\kappa_{1}$ and
$\kappa_{2}$, they can be coupled together with a sample configuration
$(\eta,\zeta)\in\{0,1\}^{\ZZ}$ in such a way that $x\mapsto\eta(x)-\zeta(x)$
changes sign at most once. Comparison to the known
invariant measures allows the characterization to go through. In a somewhat similar fashion, our \cref{thm:no-v-shaped} relies on \cref{eq:zeroval-stationary} for the jointly stationary initial condition.

There are also analogies between our \cref{thm:btflucts,thm:zeroval-stationarity,thm:zeroval-flat} and
previous work on ASEP. The shock location $b_{t}$
is analogous to the location of a second-class particle in ASEP; this connection
was first shown at the level of hydrodynamic limits in \cite{Ferrari1992}.
Later, Ferrari and Fontes showed in \cite{Ferrari-Fontes-1994b} that
the trajectory of the second-class particle in a shock-like configuration
converges, after a diffusive scaling, to Brownian motion. This is
related to our result \cref{eq:stationary-bt}. We note that
it is not an exact analogue, since our initial shock profile is a
transformation of jointly invariant measures with different drifts,
so the configurations to the left and right of the origin are not
independent. Our proof is also quite different: we use explicit calculations
from the description of the jointly invariant measures for the KPZ
equation given in \cite{GRASS-23}, while \cite{Ferrari-Fontes-1994b}
uses combinatorial calculations that are accessible only in the discrete
setting. Many of these combinatorial calculations come from the earlier
work \cite{Ferrari-Fontes-1994}.

In the case of flat initial data, an analogue of \cref{eq:flat-bt}
was proved in \cite{Ferrari-Ghosal-Nejjar-2019}. The analogy
is again not perfect, since that work considered a zero-temperature/inviscid
setting (TASEP), but in this case the proof techniques are more similar.
Those authors started from the distributional equality between the
trajectory of the second-class particle in TASEP and the competition
interface in exponential last-passage percolation \cite{ferr-pime-05}.
They then used the known convergence of the one-point distribution
of TASEP from flat initial condition to the Tracy--Widom GOE distribution
proved in \cite{Ferrari-Occelli-2018}, although with decorrelation
results \cite{Corwin-Ferrari-Peche-2012,Ferrari2008,Ferrari-Nejjar-2014}
to get independence of the GOE random variables. In our setting, we
use convergence of the KPZ equation to the KPZ fixed point \cite{Wu-23}
to get the GOE convergence, and then use localization estimates from
\cite{Das-Zhu-22b} to obtain the independence. An additional important
ingredient is an identity for the weight function of the continuum
directed random polymer in the half space in terms of the stochastic
heat equation with Dirichlet boundary conditions (\cref{lem:hlSHE}),
which is intuitive but which we could not find in the literature.

\subsection{Organization of the paper}

In \cref{sec:preliminaries}, we introduce some notation and
function spaces, and then summarize results from the literature that
are important to our techniques. In \cref{sec:zerovalflucts},
we consider the fluctuations of $h_{+}(t,0)-h_{-}(t,0)$, proving
\cref{thm:zeroval-stationarity,thm:zeroval-flat}
as well as \cref{eq:shockreferenceframe-zerodiff} of \cref{thm:btflucts}(\ref{enu:tiltedic}).
In \cref{sec:V-shaped}, we study the behavior of V-shaped solutions,
proving \cref{thm:no-v-shaped,thm:whatcanyouconvergeto}.
Finally, in \cref{sec:btflucts}, we study the fluctuations of
$b_{t}$, completing the proof of \cref{thm:btflucts}.

\subsection{Funding}
E.S. was partially supported by the Fernholz Foundation. Part of this
work was completed during the workshop ``Universality and Integrability
in KPZ'' at Columbia University, March 15--19, 2024, which was supported
by NSF grants DMS-2400990 and DMS-1664650.

\subsection{Conflict of interest statement}
The authors have no conflicts of interest to declare.

\subsection{Data availability statement} There is no data associated to this manuscript.

\subsection{Acknowledgements}
E.S. wishes to thank Timo Seppäläinen for several helpful discussions, M\'arton Bal\'azs for helpful discussions and pointers to the literature,
and Ivan Corwin for pointers to the literature, several general discussions,
and helpful discussions related to the proof of \cref{lem:KPZmax}.
A.D. would like to thank Yu Gu for encouragement and helpful discussions.
The authors also thank Sayan Das for helpful comments on an early version of the manuscript, and Xuan Wu for helpful discussions about the paper \cite{Wu-23}. 
Finally, the authors are very grateful to two anonymous referees for carefully reading
the manuscript and pointing out several important issues, which have now been corrected.

\section{Preliminaries}\label{sec:preliminaries}

In this section we review known results on the solution theory of
the KPZ equation on the whole line, and in particular introduce some
notation we will use. We use the framework of \cite{Alb-Janj-Rass-Sepp-22},
and we will largely follow their notation. In addition, we will introduce
some function spaces related to V-shaped solutions adapted from \cite{Dunlap-Ryzhik-2020}
(which works in terms of the derivative process and so uses somewhat
different, although largely equivalent, notations).

\subsection{Notational conventions}
\begin{enumerate}
\item We write $G(t,x)=\frac{1}{\sqrt{2\pi t}}\e^{-x^{2}/(2t)}$ for the
standard heat kernel.
\item For a topological space $\mathcal{Z}$, we write
$\mathcal{C}_{\mathrm{b}}(\mathcal{Z})$ for the set of bounded continuous
functions on $\mathcal{Z}$.
\item We denote equality in distribution by $\overset{\mathrm{law}}{=}$.
\item For a function $f\colon\RR\to\RR$, we define the spatial translation
  \begin{equation}
    \tau_{x}f(y)=f(x+y).\label{eq:taudef}
\end{equation}
We also define the horizontal centering
\begin{equation}
\pi_{x}f(y)=f(x+y)-f(x).\label{eq:pidef}
\end{equation}
\item For a $k$-tuple of functions $\mathbf{f}=(f_{1},\ldots,f_{k})$,
we define $\tau_{x}\mathbf{f}$ and $\pi_{x}\mathbf{f}$ to be the coordinatewise applications of $\tau_{x}$ and $\pi_{x}$, respectively.
\end{enumerate}

\subsection{Function spaces}

Here we define the spaces in which we solve the KPZ equation, following
\cite[(1.4), (1.6), and (1.11)]{Alb-Janj-Rass-Sepp-22}. We define
\begin{equation}
\mathcal{M}_{\mathrm{HE}}\coloneqq\left\{ \mu\text{ a positive Borel measure on }\RR\st\int_{\RR}\e^{-ax^{2}}\mu(\dif x)<\infty\text{ for all }a>0\right\} ,\label{eq:MHEdef}
\end{equation}
\begin{equation}
\mathcal{C}_{\mathrm{HE}}\coloneqq\left\{ f\in\mathcal{C}(\RR;(0,\infty))\st\int_{\RR}\e^{-ax^{2}}f(x)\,\dif x<\infty\text{ for all }a>0\right\} ,\label{eq:CHEdef}
\end{equation}
and
\begin{equation}
\mathcal{C}_{\mathrm{KPZ}}\coloneqq\{\log\circ f\st f\in\mathcal{C}_{\mathrm{HE}}\}=\left\{ f\in\mathcal{C}(\RR)\st\int_{\RR}\e^{f(x)-ax^{2}}\,\dif x<\infty\text{ for all }a>0\right\} .\label{eq:CKPZdef}
\end{equation}
We use the topology on $\mathcal{C}_{\mathrm{HE}}$ induced by uniform
convergence on compact sets as well as convergence of integrals of
the form $\int_{\RR}\e^{-ax^{2}}f(x)\,\dif x$. The topology
on $\mathcal{C}_{\mathrm{KPZ}}$ is such that the map $(\log\circ)\colon\mathcal{C}_{\mathrm{HE}}\to\mathcal{C}_{\mathrm{KPZ}}$
is a homeomorphism. It was shown in \cite{Alb-Janj-Rass-Sepp-22}
that $\mathcal{C}_{\mathrm{KPZ}}$ is a Polish space.

As we have noted in the introduction, there are no invariant probability
measures for the KPZ dynamics on $\mathcal{C}_{\mathrm{KPZ}}$, since
the fluctuations of $h(t,0)$ will grow as $t\to\infty$. To consider
invariant measures, we define the space
\begin{equation}
\mathcal{C}_{\mathrm{KPZ};0}\coloneqq\{f\in\mathcal{C}_{\mathrm{KPZ}}\st f(0)=0\}.\label{eq:CKPZ0def}
\end{equation}
Recalling the definition \cref{eq:pidef}, we note that, for
each $x\in\RR$, the map $\pi_{x}$ maps $\mathcal{C}_{\mathrm{KPZ}}$
to $\mathcal{C}_{\mathrm{KPZ};0}$.

We have also noted in the introduction that, in studying V-shaped
solutions to the KPZ equation, it is helpful to construct them from
pairs of solutions. We now introduce some useful function spaces
for considering pairs of solutions and V-shaped solutions to the KPZ equation. For $\theta>0$,
we define
\begin{equation}
\mathcal{Y}(\theta)\coloneqq\left\{ (f_{-},f_{+})\in\mathcal{C}_{\mathrm{KPZ}}^{2}\st\lim_{|x|\to\infty}\frac{f_{\pm}(x)}{x}=\pm\theta\right\} \label{eq:Ythetadef}
\end{equation}
and
\begin{equation}
\mathcal{Y}_{0}(\theta)\coloneqq\mathcal{Y}(\theta)\cap\mathcal{C}_{\mathrm{KPZ};0}^{2}.\label{eq:Y0thetadef}
\end{equation}
We further define
\begin{equation}
\mathcal{X}(\theta)\coloneqq\left\{ (f_{-},f_{+})\in\mathcal{Y}(\theta)\st f_{+}-f_{-}\text{ is strictly increasing}\right\} \label{eq:Xthetadef}
\end{equation}
and
\begin{equation}
\mathcal{X}_{0}(\theta)\coloneqq\mathcal{X}(\theta)\cap\mathcal{C}_{\mathrm{KPZ};0}^{2}.\label{eq:X0thetadef}
\end{equation}
Finally, we define a space of V-shaped functions with asymptotic slopes
$\pm\theta$:
\begin{equation}
\mathcal{V}(\theta)\coloneqq\left\{ f\in\mathcal{C}_{\mathrm{KPZ}}\st\lim_{|x|\to\infty}\frac{f(x)}{|x|}=\theta\right\} .\label{eq:Vthetadef}
\end{equation}
As in \cref{eq:Vshapedsolution}, we define the map $V\colon\mathcal{Y}(\theta)\to\mathcal{V}(\theta)$
by
\begin{equation}
V[f_{-},f_{+}](x)\coloneqq\log\frac{\e^{f_{+}(x)}+\e^{f_{-}(x)}}{2}.\label{eq:Vdef}
\end{equation}
It is straightforward to check that the spaces $\mathcal{Y}(\theta)$,
$\mathcal{Y}_{0}(\theta)$, $\mathcal{X}(\theta)$, and $\mathcal{X}_{0}(\theta)$
are all Borel-measurable subsets of the space $\mathcal{C}_{\mathrm{KPZ}}^{2}$,
and that $\mathcal{V}(\theta)$ is a Borel-measurable subset of $\mathcal{C}_{\mathrm{KPZ}}$.
We equip all of these spaces with the subspace topologies induced
by the respective inclusions.

\subsection{The KPZ dynamics}

We let $Z(t,x\viiva s,y)$ denote the fundamental solution to the multiplicative
stochastic heat equation \cref{eq:SHE}. It satisfies
\begin{align*}
\dif_{t}Z(t,x\viiva s,y) & =\frac{1}{2}\Delta_{x}Z(t,x\viiva s,y)\dif t+Z(t,x\viiva s,y)\dif W(t,x), & -\infty<s<t<\infty\text{ and }x,y\in\RR;\\
Z(t,x\viiva t,y) & =\delta(x-y), & t,x,y\in\RR.
\end{align*}
This process was constructed (simultaneously for all $t,x,s,y$ on
a single event of probability $1$) in \cite{Alberts-Khanin-Quastel-2014a};
see also \cite{Alb-Janj-Rass-Sepp-22}. We define the (``physical'')
solution to \cref{eq:KPZ} with initial data $h(s,\cdot)\in\mathcal{C}_{\mathrm{KPZ}}$
at time $s$ by
\[
h(t,x)=\log\int_{\RR}Z(t,x\viiva s,y)\e^{h(s,y)}\,\dif y,\qquad t>s.
\]
Then $h(t,\cdot)\in\mathcal{C}_{\mathrm{KPZ}}$ for all $t>s$ according
to the results of \cite[§2.1]{Alb-Janj-Rass-Sepp-22}.

For our applications, it will be important that certain projections
of the KPZ dynamics are Markov processes whose semigroups satisfy
the Feller property.
\begin{prop}
\label{prop:projected-feller}Let $N\in\NN$ and let $g\colon\mathcal{C}_{\mathrm{KPZ}}^{N}\to\RR^{N}$
be a continuous linear map such that $g[x\mapsto g[\mathbf{f}]]\equiv g[\mathbf{f}]$
for all $\mathbf{f}\in\mathcal{C}_{\mathrm{KPZ}}^{N}$. (Here, $x \mapsto g[\mathbf f]$ denotes the constant function with value $g[\mathbf f]$.) Define $\pi\colon\mathcal{C}_{\mathrm{KPZ}}^{N}\to\mathcal{C}_{\mathrm{KPZ}}^{N}$
by $\pi[\mathbf{f}](x)=\mathbf{f}(x)-g[\mathbf{f}]$.
\begin{enumerate}
\item For any vector $\mathbf{h}=(h_{1},\ldots,h_{N})$ of solutions to
\cref{eq:KPZ}, the process $(\pi[\mathbf{h}(t,\cdot)])_{t\ge0}$
is a Markov process with state space $\mathcal{C}_{\mathrm{KPZ}}^{N}$.
\item For $F\in\mathcal{C}_{\mathrm{b}}(\mathcal{C}_{\mathrm{KPZ}}^{N})$,
$t\ge0$, and $\mathbf{f}\in\mathcal{C}_{\mathrm{KPZ}}^{N}$, let
$P_{t}^{\pi}F(\mathbf{f})=\EE[F[\pi[\mathbf{h}(t,\cdot)]]]$,
where $\mathbf{h}$ is a vector of solutions to \cref{eq:KPZ}
with initial condition $\mathbf{h}(0,x)=\mathbf{f}(x)$. Then the
Markov semigroup $(P_{t}^{\pi})_{t\ge0}$ has the Feller property.
\end{enumerate}
\end{prop}

\begin{proof}
We fix $s<t$ and note that
\begin{align*}
\mathbf{h}(t,x) & =\log\int_{\RR}Z(t,x\viiva s,y)\exp(\mathbf{h}(s,y))\,\dif y\\
 & =\log\int_{\RR}Z(t,x\viiva s,y)\exp(\pi[\mathbf{h}(s,\cdot)](y)+g[\mathbf{h}(s,\cdot)])\,\dif y\\
 & =\log\int_{\RR}Z(t,x\viiva s,y)\exp(\pi[\mathbf{h}(s,\cdot)](y))\,\dif y+g[\mathbf{h}(s,\cdot)],
\end{align*}
where $\log$ and $\exp$ act on vectors componentwise. Therefore,
using the assumptions on $g$, we have
\begin{align*}
\pi[\mathbf{h}(t,\cdot)](z) & =\pi\left[x\mapsto\log\int_{\RR}Z(t,x\viiva s,y)\exp(\pi[\mathbf{h}(s,\cdot)](y))\,\dif y+g[\mathbf{h}(s,\cdot)]\right](z)\\
 & =\pi\left[x\mapsto\log\int_{\RR}Z(t,x\viiva s,y)\exp(\pi[\mathbf{h}(s,\cdot)](y))\,\dif y\right](z).
\end{align*}
From this we see that $\pi[\mathbf{h}(t,\cdot)]$ depends only on
$\pi[\mathbf{h}(s,\cdot)]$ and the noise between $s$ and $t$, and
conclude that $(\pi[\mathbf{h}(t,\cdot)])_{t}$ is a Markov process.
The fact that $(P_{t}^{\pi})_{t\ge0}$ has the Feller property is
then an immediate consequence of the same statement for $(P_{t}^{\id})_{t\ge0}$,
which was shown in \cite[Remark 2.12]{Alb-Janj-Rass-Sepp-22}.
\end{proof}
Recall the definition \cref{eq:Vdef} of $V$.
\begin{prop}
\label{prop:Vsolves}If $h_{-}$ and $h_{+}$ are solutions to \cref{eq:KPZ},
and we define $h_{\mathsf{V}}(t,x)\coloneqq V[(h_{-},h_{+})(t,\cdot)](x)$,
then $h_{\mathsf{V}}$ is also a solution to \cref{eq:KPZ}.
\end{prop}

\begin{proof}
We note that $\e^{h_{\mathsf{V}}(t,x)}=\frac{1}{2}(\e^{h_{-}(t,x)}+\e^{h_{+}(t,x)})$,
and the conclusion follows from the linearity of the multiplicative
stochastic heat equation.
\end{proof}
The following proposition, which plays a role similar to that of \cite[Lemma~2.2]{Dunlap-Ryzhik-2020}, shows that the
space $\mathcal{X}(\theta)$ is preserved by the KPZ dynamics.
\begin{prop}
\label{prop:Xpreserved}Let $\theta>0$ and let $h_{-}$ and $h_{+}$
be solutions to \cref{eq:KPZ} with initial data $(h_{-},h_{+})(s,\cdot)\in\mathcal{X}(\theta)$.
Then we have $(h_{-},h_{+})(t,\cdot)\in\mathcal{X}(\theta)$ for all
$t>s$.
\end{prop}

\begin{proof}
Fix $t>s$. The fact that $\lim\limits_{|x|\to\pm\infty}\frac{h_{\pm}(t,x)}{x}=\pm\theta$
is proved as \cite[Proposition 2.13]{Alb-Janj-Rass-Sepp-22}, so it
remains to prove that $(h_{+}-h_{-})(t,\cdot)$ is strictly increasing.
Let $x_{1}<x_{2}$. Define
\begin{equation}
z_{ij}(y_{1},y_{2})\coloneqq Z(t,x_{i}\viiva s,y_{1})Z(t,x_{j}\viiva s,y_{2})\label{eq:zdef}
\end{equation}
and
\[
k(y_{1},y_{2})\coloneqq\exp\{h_{-}(s,y_{1})+h_{+}(s,y_{2})\},
\]
so we can write
\begin{align}
h_{+}(t,x_{2}) & -h_{-}(t,x_{2})-(h_{+}(t,x_{1})-h_{-}(t,x_{1}))=\log\frac{\iint_{\RR^{2}}z_{12}(y_{1},y_{2})k(y_{1},y_{2})\,\dif y_{1}\,\dif y_{2}}{\iint_{\RR^{2}}z_{21}(y_{1},y_{2})k(y_{1},y_{2})\,\dif y_{1}\,\dif y_{2}}\nonumber \\
 & =\log\frac{\iint_{y_{1}<y_{2}}[z_{12}(y_{1},y_{2})k(y_{1},y_{2})+z_{12}(y_{2},y_{1})k(y_{2},y_{1})]\,\dif y_{1}\,\dif y_{2}}{\iint_{y_{1}<y_{2}}[z_{21}(y_{1},y_{2})k(y_{1},y_{2})+z_{21}(y_{2},y_{1})k(y_{2},y_{1})]\,\dif y_{1}\,\dif y_{2}}\nonumber \\
 & =\log\frac{\iint_{y_{1}<y_{2}}[z_{12}(y_{1},y_{2})k(y_{1},y_{2})+z_{12}(y_{2},y_{1})k(y_{2},y_{1})]\,\dif y_{1}\,\dif y_{2}}{\iint_{y_{1}<y_{2}}[z_{12}(y_{2},y_{1})k(y_{1},y_{2})+z_{12}(y_{1},y_{2})k(y_{2},y_{1})\,\dif y_{1}\,\dif y_{2}},\label{eq:hmonotone-prep}
\end{align}
where in the last identity we used that 
\[
z_{21}(w_{1},w_{2})=z_{12}(w_{2},w_{1})
\]
for any $w_{1},w_{2}\in\RR$ by the definition \cref{eq:zdef}.
Now we have, whenever $y_{1}<y_{2}$, that
\[
z_{12}(y_{1},y_{2})>z_{12}(y_{2},y_{1})
\]
by \cite[Theorem 2.17]{Alb-Janj-Rass-Sepp-22} and 
\[
k(y_{1},y_{2})>k(y_{2},y_{1})
\]
by the assumption that $(h_{+}-h_{-})(s,\cdot)$ is strictly increasing.
This implies that
\begin{align*}
z_{12} & (y_{1},y_{2})k(y_{1},y_{2})+z_{12}(y_{2},y_{1})k(y_{2},y_{1})-[z_{12}(y_{2},y_{1})k(y_{1},y_{2})+z_{12}(y_{1},y_{2})k(y_{2},y_{1})]\\
 & =[z_{12}(y_{1},y_{2})-z_{12}(y_{2},y_{1})]\cdot[k(y_{1},y_{2})-k(y_{2},y_{1})]>0
\end{align*}
whenever $y_{1}<y_{2}$, and so the right side of \cref{eq:hmonotone-prep}
is positive, which is what we wanted to show.
\end{proof}
In the following sections, we will also make frequent use of the scaling
relations of the KPZ equation, or equivalently of the stochastic heat
equation. We cite a result from \cite{Alb-Janj-Rass-Sepp-22}, which
gives a full distributional equality for the four-parameter process
$Z$. At the level of the KPZ equation, these have been previously
well-known. We only state the invariances we need for our purposes.
\begin{prop}[{\cite[Lemma 3.1]{Alb-Janj-Rass-Sepp-22}}] \label{prop:invariances} 
The process $Z(t,x\viiva s,y)$ satisfies the following scaling invariances
as a process in the space $\mathcal{C}(\RR_{\uparrow}^{4};\RR)$,
where $\RR_{\uparrow}^{4}\coloneqq\{(t,x,s,y)\in\RR^{4}\st s<t\}$.
\begin{description}[leftmargin=0pt]
\item [{(Shift)}] For $u,z\in\RR$, we have
\begin{equation}
\left\{ Z(t+u,x+z\viiva s+u,y+z)\right\} _{(t,x,s,y)\in\RR_{\uparrow}^{4}}\overset{\mathrm{law}}{=}\left\{ Z(t,x\viiva s,y)\right\} _{(t,x,s,y)\in\RR_{\uparrow}^{4}}.\label{eq:Z_shift}
\end{equation}
\item [{(Reflection)}] We have
\begin{equation}
\left\{ Z(t,x\viiva s,y)\right\} _{(t,x,s,y)\in\RR_{\uparrow}^{4}}\overset{\mathrm{law}}{=}\left\{ Z(t,-x\viiva s,-y)\right\} _{(t,x,s,y)\in\RR_{\uparrow}^{4}}\overset{\mathrm{law}}{=}\left\{ Z(-s,y\viiva -t,x)\right\} _{(t,x,s,y)\in\RR_{\uparrow}^{4}}.\label{eq:Z_reflect}
\end{equation}
\item [{(Shear)}] For each $r,\nu\in\RR^{2}$, we have
\begin{equation}
\left\{ \e^{\nu(x-y)+\frac{\nu^{2}}{2}(t-s)}Z(t,x+\nu(t-r)\viiva s,y+\nu(s-r))\right\} _{(t,x,s,y)\in\RR_{\uparrow}^{4}}\overset{\mathrm{law}}{=}\left\{ Z(t,x\viiva s,y)\right\} _{(t,x,s,y)\in\RR_{\uparrow}^{4}}.\label{eq:Z_shear}
\end{equation}
 
\end{description}
\end{prop}

\begin{rem} \label{rmk:shear}
It is a consequence of \cref{eq:Z_shear} that, if $\theta\in\RR$
and $h_{\theta}$ and $h_{0}$ each solve \cref{eq:KPZ} with
$h_{\theta}(0,x)=h_{0}(0,x)+\theta x$, then
\begin{equation}
\left\{ h_{\theta}(t,x-\theta t)\right\} _{(t,x)\in\RR_{+}\times\RR}\overset{\mathrm{law}}{=}\left\{ h_{0}(t,x)+\theta x-\frac{\theta^{2}}{2}t\right\} _{(t,x)\in\RR_{+}\times\RR}.\label{eq:h_shear}
\end{equation}
To see this from \cref{eq:Z_shear}, note that
\begin{align*}
h_{\theta}(t,x-\theta t) & =\log\int_{\RR}Z(t,x-\theta t\viiva0,y)\e^{h_{0}(0,y)+\theta y}\,\dif y\\
\overset{\mathrm{law}}&{=}-\frac{\theta^{2}}{2}t+\theta x+\log\int_{\RR}Z(t,x\viiva0,y)\e^{h_{0}(y)}\,\dif y=-\frac{\theta^{2}}{2}t+\theta x+h_{0}(t,x),
\end{align*}
and indeed the distributional equality holds as processes in $(t,x)\in\RR_{+}\times\RR$.
\end{rem}

Finally, we will use the following estimate from \cite{Janj-Rass-Sepp-22}:
\begin{lem}[{\cite[Lemma 6.6]{Janj-Rass-Sepp-22}}]
\label{lem:shapefn}The following holds with probability $1$. For
all $\theta\in\RR$, all $-\infty\le\lambda_{1}<\lambda_{2}\le\infty$,
and all $C<\infty$,
\[
\adjustlimits\lim_{t\to+\infty}\sup_{r,y\in[-C,C]}\left|\frac{1}{t}\log\int_{\lambda_{1}t}^{\lambda_{2}t}Z(t+r,y\viiva0,x)\e^{\theta x}\,\dif x-\sup_{\lambda_{1}<\lambda<\lambda_{2}}\left\{ \theta\lambda-\frac{\lambda^{2}}{2}-\frac{1}{24}\right\} \right|=0.
\]
\end{lem}

\subsection{Stationarity properties}\label{subsec:stationarity-properties}

We now turn our attention to what is known about the ergodic theory
of the KPZ equation. First we recall the single-$\theta$ stationary
solutions.
\begin{defn}
\label{def:muthetadef}For $\theta\in\RR$, we let $\mu_{\theta}$
be the law of $x\mapsto B(x)+\theta x$, where $B$ is a standard
two-sided Brownian motion with $B(0)=0$.
\end{defn}

It is clear from the definitions and standard properties of Brownian
motion that
\[
\mu_{\theta}(\mathcal{C}_{\mathrm{KPZ};0})=1
\]
(recalling the definition \cref{eq:CKPZ0def}). The law $\mu_{\theta}$
is invariant for the recentered KPZ dynamics, as we state in the following
proposition. Recall the definition \cref{eq:pidef} of $\pi_{0}$.
\begin{prop}
\label{prop:mu-invariant}If $h$ solves \cref{eq:KPZ} with
initial condition $h(0,\cdot)\sim\mu_{\theta}$ independent of the
noise, then $\pi_{0}[h(t,\cdot)]\sim\mu_{\theta}$ for each $t>0$
as well.
\end{prop}

\cref{prop:mu-invariant} was proved for $\theta=0$ in \cite[Proposition B.1]{Bertini-Giacomin-1997},
and the result for general $\theta$ follows from the shear-invariance
\cref{eq:h_shear}. See also \cite[Theorem 1.2]{Funaki-Quastel-2015}
and \cite[Theorem 3.26(i)]{Janj-Rass-Sepp-22}.

Next, we consider jointly stationary solutions to \cref{eq:KPZ}.
These were considered in \cite{GRASS-23}, and we now review the results
proved there that we will need.

Let $\theta>0$. Consider the mapping $\mathcal{D}\colon\mathcal{Y}_{0}(\theta)\to\mathcal{X}_{0}(\theta)$
defined by
\begin{equation}
\mathcal{D}[f_{-},f_{+}](x)\coloneqq\left(f_{-}(x),f_{+}(x)+\log\frac{\int_{-\infty}^{x}e^{(f_{+}(y)-f_{+}(x))-(f_{-}(y)-f_{-}(x))}\,\dif y}{\int_{-\infty}^{0}e^{f_{+}(y)-f_{-}(y)}\,\dif y}\right).\label{eq:Ddef}
\end{equation}
That the function $\mathcal{D}$ in fact takes $\mathcal{Y}_{0}(\theta)$
to $\mathcal{X}_{0}(\theta)$ is proved in \cite[Lemmas 2.2–2.3]{GRASS-23}.
The following is a restatement of the definition \cref{eq:nuthetadef}
of $\nu_{\theta}$ given in the introduction.
\begin{defn}
We denote by $\nu_{\theta}$ the law of $\mathcal{D}[B_{1}(\cdot)-\theta\cdot,B_{2}(\cdot)+\theta\cdot]$,
where $B_{1},B_{2}$ are independent two-sided Brownian motions with
$B_{1}(0)=B_{2}(0)=0$.
\end{defn}

Since $(B_{1}(\cdot)-\theta\cdot,B_{2}(\cdot)+\theta\cdot)$ is evidently
an element of $\mathcal{Y}_{0}(\theta)$ with probability $1$, and
$\mathcal{D}$ maps $\mathcal{Y}_{0}(\theta)$ to $\mathcal{X}_{0}(\theta)$
as observed above, we have
\begin{equation}
\nu_{\theta}(\mathcal{X}_{0}(\theta))=1.\label{eq:alwaysinX0}
\end{equation}
We also note that
\begin{equation}
\mathcal{D}[B_{1}(\cdot)-\theta\cdot,B_{2}(\cdot)+\theta\cdot](x)=\left(B_{1}(x)-\theta x,B_{2}(x)+\theta x+\mathcal{S}_{\theta}(x)-\mathcal{S}_{\theta}(0)\right),\label{eq:DBs}
\end{equation}
where
\begin{equation}
\mathcal{S}_{\theta}(x)=\log\int_{-\infty}^{x}\exp\left\{ (B_{2}(y)-B_{2}(x))-(B_{1}(y)-B_{1}(x))+2\theta(y-x)\right\} \,\dif y.\label{eq:Sdef}
\end{equation}
We further observe that the process $(\mathcal{S}_{\theta}(x))_{x}$
is stationary in space.

We now recall the result of \cite{GRASS-23} that $\nu_{\theta}$
is an invariant measure for the spatial increments of the KPZ equation.
\begin{prop}[{\cite[Theorem 1.1]{GRASS-23}}]
\label{prop:joint-stationarity}Suppose that $\mathbf{h}=(h_{-},h_{+})$
is a vector of two solutions to \cref{eq:KPZ} for $t>s$ with
$\mathbf{h}(s,\cdot)\sim\nu_{\theta}$ (independent of the noise).
Then, for each $t>s$, we have $\pi_{0}[\mathbf{h}(t,\cdot)]\sim\nu_{\theta}$
as well.
\end{prop}

Finally, we address the stability/convergence properties of the measures
$\nu_{\theta}$. Again, we only state the convergence result that
we need.
\begin{prop}[{\cite{Janj-Rass-Sepp-22,GRASS-23}}]
\label{prop:spatially-homog-convergence}  For any $\theta>0$,  there is
a random process $\overline{\mathbf{f}}=(\overline{f}_{-},\overline{f}_{+})\sim\nu_{\theta}$
such that the following holds with probability one. For any $\mathbf{f}=(f_{-},f_{+})\in\mathcal{Y}(\theta)$, let $\mathbf{h}^{T}$ be a vector of solutions
to \cref{eq:KPZ} with initial data $\mathbf{h}^{T}(-T,\cdot)=\mathbf{f}$.
Then we have the convergence
\begin{equation}\label{eq:convergetofbar}
\lim_{T\to\infty}\pi_{0}[\mathbf{h}^{T}(0,\cdot)]=\overline{\mathbf{f}}.
\end{equation}
 in the topology of $\mathcal{C}_{\mathrm{KPZ};0}$.

As a consequence of this and the temporal invariance of the KPZ equation,
we see that if $\mathbf{h}$ is a vector of solutions to \cref{eq:KPZ}
with initial data $\mathbf{h}(0,\cdot)=\mathbf{f}$, then $\pi_{0}[\mathbf{h}(t,\cdot)]$
converges in distribution to $\overline{\mathbf{f}}$ as $t\to\infty$.
\end{prop}
\begin{proof}
The existence of an $\overline{\mathbf{f}}$ such that the convergence \cref{eq:convergetofbar} holds uniformly on compact sets is \cite[Theorem~3.23]{Janj-Rass-Sepp-22}. That the convergence in fact holds in the topology of $\mathcal{C}_{\mathrm{KPZ};0}$ (i.e.\ that all integrals of the form $\int_\RR \e^{-ax^2+ h^T_{\pm}(0,x)}\,\dif x$, with $a>0$, converge) is then a consequence of the dominated convergence theorem and \cite[Lemma~7.6]{Janj-Rass-Sepp-22}. Since the Markov process has the Feller property (\cref{prop:projected-feller}), a Krylov--Bogoliubov argument (see e.g.\ \cite[Theorem 3.1.1]{Da-Prato-Zabczyk-1996}) shows that the limit $\overline{\mathbf{f}}$ must be distributed according to a jointly invariant measure for \cref{eq:KPZ}. Also, its two components must have asymptotic slopes $\pm\theta$ by \cite[Theorem~3.1(d)]{Janj-Rass-Sepp-22}. But $\nu_\theta$ is the unique such jointly invariant measure by \cite[Theorem~1.1]{GRASS-23}, and so in fact we have $\overline{\mathbf{f}}\sim\nu_\theta$.
\end{proof}

\begin{rem}
In fact, the basin of attraction of the measure $\nu_{\theta}$ is
larger than $\mathcal{Y}(\theta)$; see the discussion after \cref{thm:whatcanyouconvergeto} and also \cite[Lemma 2.22 and Theorem 3.23]{Janj-Rass-Sepp-22}.
\end{rem}

\subsection{The shock reference frame}

In this section we prove \cref{thm:nuhat}, closely following
the proof of \cite[Theorem 1.1]{Dunlap-Ryzhik-2020}. We first introduce
some notation. For $\mathbf{f}=(f_{-},f_{+})\in\mathcal{X}(\theta)$,
we define
\begin{equation}
\mathfrak{b}[\mathbf{f}]\coloneqq(f_{+}-f_{-})^{-1}(0).\label{eq:bfdef}
\end{equation}
Then we can define
\begin{equation}
\label{eq:piShdef}
\pi_{\mathrm{Sh}}[\mathbf{f}](x)=\pi_{\mathfrak{b}[\mathbf{f}]}[\mathbf{f}](x)=\mathbf{f}(\mathfrak{b}[\mathbf{f}]+x)-\mathbf{f}(\mathfrak{b}[\mathbf{f}]).
\end{equation}
The map $\pi_{\mathrm{Sh}}$ translates the graph of $\mathbf{f}$
horizontally and vertically so that the intersection point of the
graphs of $f_{-}$ and $f_{+}$ is moved to the origin. Recall the
definitions \cref{eq:pidef} of $\pi_{x}$ and \cref{eq:Ddef}
of $\mathcal{D}$. We need a result on how these maps intertwine.
\begin{lem}
\label{lem:Dshift}For each $(f_{-},f_{+})\in\mathcal{Y}_{0}(\theta)$
and $x\in\RR$, we have
\[
\pi_{x}[\mathcal{D}[f_{-},f_{+}]]=\mathcal{D}[\pi_{x}[(f_{-},f_{+})]].
\]
\end{lem}

\begin{proof}
We write
\begin{align*}
&\pi_{x}  [\mathcal{D}[f_{-},f_{+}]](y)\\
 & =\left(f_{-}(x+y)-f_{-}(x),f_{+}(x+y)-f_{+}(x)+\log\frac{\int_{-\infty}^{x+y}\e^{(f_{+}(w)-f_{+}(x+y))-(f_{-}(w)-f_{-}(x+y))}\,\dif w}{\int_{-\infty}^{x}\e^{(f_{+}(w)-f_{+}(x))-(f_{-}(w)-f_{-}(x))}\,\dif w}\right)\\
 & =\left(f_{-}(x+y)-f_{-}(x),f_{+}(x+y)-f_{+}(x)+\log\frac{\int_{-\infty}^{y}\e^{(f_{+}(x+w)-f_{+}(x+y))-(f_{-}(x+w)-f_{-}(x+y))}\,\dif w}{\int_{-\infty}^{0}\e^{(f_{+}(x+w)-f_{+}(x))-(f_{-}(x+w)-f_{-}(x))}\,\dif w}\right)\\
 & =\mathcal{D}[\pi_{x}[f_{-},f_{+}]](y).\qedhere
\end{align*}
\end{proof}
Now we can prove the following using ergodicity.
\begin{lem}
\label{lem:nuhatexpectation}Let $F\in\mathcal{C}_{\mathrm{b}}(\mathcal{C}_{\mathrm{KPZ};0}^{2})$ and let $\mathbf{f}=(f_-,f_+)$ be an $\mathcal{X}_0(\theta)$-valued random variable for some $\theta > 0$.
Let $\EE_{\nu_{\theta}}$ denote expectation under which $\mathbf{f}=(f_{-},f_{+})\sim\nu_{\theta}$
and let $\EE_{\hat{\nu}_{\theta}}$ denote expectation under
which $\mathbf{f}=(f_{-},f_{+})\sim\hat{\nu}_{\theta}$. Then we have
\begin{equation}
\lim_{L\to\infty}\fint_{0}^{L}\EE_{\nu_{\theta}}\left[F(\pi_{\mathrm{Sh}}[f_{-},f_{+}-\zeta])\right]\,\dif\zeta=\EE_{\hat{\nu}_{\theta}}[F(\mathbf{f})].\label{eq:nuhatexpectation}
\end{equation}
 
\end{lem}

In \cref{eq:nuhatexpectation} and henceforth, we use the notation $\fint_{0}^{L}\coloneqq\frac{1}{L}\int_{0}^{L}$. On the left side of \cref{eq:nuhatexpectation}, we average over the probability space and also over the physical space but in a non-uniform way, since $\pi_{\mathrm{Sh}}$ is a nonlinear shift. On the right side, we average over the probability space with respect to the tilted measure $\hat{\nu}_\theta$, with the tilt corresponding to the non-uniformity of the spatial shift on the left side. To prove this statement, we will use the ergodic theorem to relate the averaging over the probability space to averaging in physical space, and then using a change of variables in physical space which corresponds to the tilt.
\begin{proof}
  Let
  $\mathbf{f}_{\zeta}=(f_{-},f_{+}-\zeta)$. Recalling the definitions \cref{eq:piShdef,eq:taudef,eq:bfdef} of $\pi_{\mathrm{Sh}}$, $\tau$,  and $\mathfrak{b}$, respectively, we observe that
\[
\pi_{\mathrm{Sh}}[\mathbf{f}_{\zeta}]=\pi_{0}[\tau_{\mathfrak{b}[\mathbf{f}_{\zeta}]}\mathbf{f}_{\zeta}]=\pi_{0}[\tau_{(f_{+}-f_{-})^{-1}(\zeta)}\mathbf{f}],
\]
so
\begin{align*}
\int_{0}^{L}F(\pi_{\mathrm{Sh}}[\mathbf{f}_{\zeta}])\,\dif\zeta&=\int_{0}^{L}F(\pi_{0}[\tau_{(f_{+}-f_{-})^{-1}(\zeta)}\mathbf{f}])\,\dif\zeta\\&=\int_{0}^{(f_{+}-f_{-})^{-1}(L)}F(\pi_{0}[\tau_{z}\mathbf{f}])\partial_{x}(f_{+}-f_{-})(z)\,\dif z,
\end{align*}
where in the last identity we made the change of variables $\zeta=(f_{+}-f_{-})(z)$
and used that $f_{+}(0)=f_{-}(0)$ since $\mathbf{f}\in\mathcal{X}_{0}(\theta)$.
Dividing by $L$, we obtain
\begin{equation}
\fint_{0}^{L}  F(\pi_{\mathrm{Sh}}[\mathbf{f}_{\zeta}])\,\dif\zeta
  =\frac{(f_{+}-f_{-})^{-1}(L)}{L}\fint_{0}^{(f_{+}-f_{-})^{-1}(L)}F(\pi_{0}[\tau_{z}\mathbf{f}])\partial_{x}(f_{+}-f_{-})(z)\,\dif z.
\label{eq:dividebyL}
\end{equation}
Now as $L\to\infty$, we have
\[
\lim_{L\to\infty}\frac{(f_{+}-f_{-})^{-1}(L)}{L}=\frac{1}{2\theta}
\]
since $\mathbf{f}\in\mathcal{X}_{0}(\theta)$ almost surely. We also have
\begin{align*}
  \lim_{L\to\infty} & \fint_{0}^{(f_{+}-f_{-})^{-1}(L)}F(\pi_{0}[\tau_{z}\mathbf{f}])\partial_{x}(f_{+}-f_{-})(z)\,\dif z\\
 & =\lim_{M\to\infty}\fint_{0}^{M}F(\pi_{0}[\tau_{z}\mathbf{f}])\partial_{x}(f_{+}-f_{-})(z)\,\dif z\\
                    &=  \lim_{M\to\infty}\fint_{0}^{M}\tau_z\left[\mathbf{f}\mapsto F(\pi_{0}[\mathbf{f}])\partial_{x}(f_{+}-f_{-})(0)\right](\mathbf{f})\,\dif z\\
 & =\EE_{\nu_{\theta}}[F(\mathbf{f})\partial_{x}(f_{+}-f_{-})(0)]
\end{align*}
$\nu_{\theta}$-a.s.~by the ergodic theorem. To be precise, we use
the spatial ergodicity of the spatial increments of the process $\mathbf{f}$,
which follows from the spatial ergodicity of Brownian motion, the
definition \cref{eq:nuthetadef} of $\nu_{\theta}$, the
shift-covariance proved in \cref{lem:Dshift}, and fact that all moments of $\partial_x(f_+-f_-)$ are finite (by \cref{eq:Rdef}ff.\ below, $\partial_x(f_+-f_-)(0)$ is a Gamma-distributed random variable) . Using these limits
in \cref{eq:dividebyL}, we see that
\[
\lim_{L\to\infty}\frac{1}{L}\int_{0}^{L}F(\pi_{\mathrm{Sh}}[\mathbf{f}_{\zeta}])\,\dif\zeta=\frac{1}{2\theta}\EE_{\nu_{\theta}}\left[F(\mathbf{f})\partial_{x}(f_{+}-f_{-})(0)]\right]\overset{\cref{eq:RNderiv}}=\EE_{\hat{\nu}_{\theta}}[F(\mathbf{f})],\qquad\text{\ensuremath{\nu_{\theta}}-a.s.,}
\]
and then the bounded convergence theorem implies \cref{eq:nuhatexpectation}. 
\end{proof}
The following lemma is immediate from the definition, but we use it several times, so we state it here for convenience.
\begin{lem} \label{lem:shift_Sh}
    For $\theta > 0$ and $\mathbf f \in \mathcal X(\theta)$, then for any constant $c \in \RR$, $\pi_{\mathrm{Sh}}[\mathbf f] = \pi_{\mathrm{Sh}}[\mathbf f + (c,c)]$.
\end{lem}

Before proving Theorem \cref{thm:nuhat}, we prove one more intermediate lemma, which gives idempotence of the shift map under the KPZ equation evolution.
\begin{lem} \label{lem:idempotence}
     Let $\theta > 0$, and let $\zeta \in \RR$. Let $\mathbf f = (f_-,f_+)$ be random initial data independent of the noise such that $\mathbf f \in \mathcal X(\theta)$ almost surely. Let $\mathbf h$ denote the solution to the KPZ equation with $\mathbf h(0,x) = \mathbf f$, and let $\mathbf h_{\mathrm{Sh}}$ denote the solution to the KPZ equation with $\mathbf h_{\mathrm{Sh}}(0,x) = \pi_{\mathrm{Sh}}[\mathbf f]$. Then, for all $t > 0$,
     \[
     \Law(\pi_{\mathrm{Sh}}[\mathbf h(t,\cdot)]) = \Law(\pi_{\mathrm{Sh}}[\mathbf h_{\mathrm{Sh}}(t,\cdot)]).
     \]
\end{lem}
\begin{proof}
For $t \ge 0$, let $b_t = \mathfrak b[\mathbf h(t,\cdot)]$. Then we have
\begin{equation} \label{eq:h_shift0}
\begin{aligned}
&\mathbf h_{\mathrm{Sh}}(t,\cdot)   \\
&= \Biggl(\log \int_\RR e^{f_-(y + b_0) - f_-(b_0)}Z(t,\cdot\viiva s,y)\,dy,\log \int_\RR e^{f_+(y + b_0) - f_+(b_0)}Z(t,\cdot\viiva s,y)\,dy\Biggr) \\
&= \Biggl(\log \int_\RR e^{f_-(y)}Z(t,\cdot\viiva s,y - b_0)\,dy,\log \int_\RR e^{f_+(y)}Z(t,\cdot\viiva s,y - b_0)\,dy\Biggr) - (f_-(b_0),f_+(b_0)) \\\overset{\mathrm{law}}&{=}\Biggl(\log \int_\RR e^{f_-(y)}Z(t,\cdot + b_0\viiva s,y)\,dy,\log \int_\RR e^{f_+(y)}Z(t,\cdot + b_0\viiva s,y)\,dy\Biggr) - (f_-(b_0),f_+(b_0)) \\
&= \mathbf h(t,\cdot + b_0) - (f_-(b_0),f_+(b_0)), 
\end{aligned}
\end{equation}
where the distributional equality follows by the shift-invariance in \cref{eq:Z_shift}.

By definition of the operator $\mathfrak b$, we see immediately that $\mathfrak b[\mathbf h(t,\cdot + b_0)] = b_t - b_0$, which implies
\begin{equation} \label{eq:h_shift1}
\pi_{\mathrm{Sh}}[\mathbf h(t,\cdot + b_0)] = \mathbf h(\cdot + b_t - b_0 + b_0) - \mathbf h(b_t - b_0 + b_0) = \pi_{\mathrm{Sh}}[\mathbf h(t,\cdot)].
\end{equation}
Now, by definition of $b_0$, we have $f_0(b_0) = f_+(b_0)$, so from \eqref{eq:h_shift0}, \eqref{eq:h_shift1}, and Lemma \ref{lem:shift_Sh}, we obtain
\[
\pi_{\mathrm{Sh}}[\mathbf h_{\mathrm{Sh}}(t,\cdot)]\overset{\mathrm{law}}{=} \pi_{\mathrm{Sh}}[\mathbf h(t,\cdot + b_0) - (f_-(b_0),f_+(b_0))] = \pi_{\mathrm{Sh}}[\mathbf h(t,\cdot)]. \qedhere
\]
\end{proof}
Now, we can prove \cref{thm:nuhat}.
\begin{proof}[Proof of \cref{thm:nuhat}.]
Let $F\in\mathcal{C}_{\mathrm{b}}(\mathcal{C}_{\mathrm{KPZ};0}^{2})$.
Let $\EE$ and $\hat{\EE}$ denote expectation under
which $\mathbf{h}(0,\cdot)$ is distributed according to $\nu_{\theta}$
and $\hat{\nu}_{\theta}$, respectively, in both cases independent
of the noise. Furthermore, let $\EE_{\zeta,\mathrm{Sh}}$ denote expectation under which $\mathbf h(0,x) = \pi_{\mathrm{Sh}}[(f_-(x),f_+(x) - \zeta)]$ where $(f_-,f_+) \sim \nu_\theta$, and let $\EE_\zeta$ denote expectation under which $\mathbf h(0,x) = (f_-(x),f_+(x) - \zeta)$ where $(f_-,f_+) \sim \nu_\theta$, both independent of the noise. We seek to prove that, for any $t>0$, we have
\[
\hat{\EE}[F(\pi_{\mathrm{Sh}}[\mathbf{h}(t,\cdot)])] = \hat{\EE}[F(\mathbf{h}(0,\cdot))].
\]
We can compute
\begin{align*}
\hat{\EE} & [F(\pi_{\mathrm{Sh}}[\mathbf{h}(t,\cdot)])]=\lim_{L\to\infty}\fint_{0}^{L}\EE_{\zeta,\mathrm{Sh}}[F(\pi_{\mathrm{Sh}}[\mathbf{h}(t,\cdot)])]\,\dif\zeta\\
& = \lim_{L\to\infty}\fint_{0}^{L}\EE_\zeta[F(\pi_{\mathrm{Sh}}[\mathbf{h}(t,\cdot)]\,\dif\zeta \\& =\lim_{L\to\infty}\fint_{0}^{L}\EE[F(\pi_{\mathrm{Sh}}[h_{-}(t,\cdot),h_{+}(t,\cdot)-\zeta])]\,\dif\zeta\\ 
 & =\lim_{L\to\infty}\EE\left[\fint_{0}^{L}F(\pi_{\mathrm{Sh}}[\pi_{0}[\mathbf{h}(t,\cdot)]+(0,h_{+}(t,0)-h_{-}(t,0)-\zeta)])\,\dif\zeta\right]\\
 & =\lim_{L\to\infty}\EE\left[\fint_{-(h_{+}(t,0)-h_{-}(t,0))}^{L-(h_{+}(t,0)-h_{-}(t,0))}F(\pi_{\mathrm{Sh}}[\pi_{0}[\mathbf{h}(t,\cdot)]+(0,-\zeta)])\,\dif\zeta\right]\\
 & =\lim_{L\to\infty}\fint_{0}^{L}\EE\left[F(\pi_{\mathrm{Sh}}[\pi_{0}[\mathbf{h}(t,\cdot)]+(0,-\zeta)])\right]\,\dif\zeta\\
 & =\lim_{L\to\infty}\fint_{0}^{L}\EE\left[F(\pi_{\mathrm{Sh}}[\mathbf{h}(0,\cdot)+(0,-\zeta)])\right]\,\dif\zeta\\
 & =\hat{\EE}[F(\mathbf{h}(0,\cdot))].
\end{align*}
The first identity is by \cref{lem:nuhatexpectation}, the second is by the idempotence in Lemma \ref{lem:idempotence} (applied to initial data $(f_-(x),f_+(x) - \zeta)$), the third is by the fact that the KPZ equation commutes with height shifts of the initial data, the fourth
is by the definition of $\pi_0$ and Lemma \ref{lem:shift_Sh}, the fifth is by a change of variables,
the sixth is by the ergodic theorem, the seventh is by \cref{prop:mu-invariant},
and the last is by \cref{lem:nuhatexpectation} again.
\end{proof}

\section{Fluctuations of differences of KPZ solutions at the origin}\label{sec:zerovalflucts}

The results of \cref{subsec:stationarity-properties} described
the (stationary) fluctuations of $\pi_{0}[\mathbf{h}(t,\cdot)]=\mathbf{h}(t,\cdot)-\mathbf{h}(t,0)$
for $\mathbf{h}$ a vector of solutions to \cref{eq:KPZ}. Not
captured in these results is the behavior of $\mathbf{h}(t,0)$, as
this is exactly what is forgotten by $\pi_{0}$. In this section,
we consider these results both in the setting of stationary initial
data and of flat initial data.

\subsection{Stationary case (static reference frame)}

In this section we prove \cref{thm:zeroval-stationarity}.
\begin{proof}[Proof of \cref{thm:zeroval-stationarity}]
The proof proceeds in two steps. First, we will show that
\begin{equation}
t^{-1/2}[h_{+}(t,-\theta t)-h_{-}(t,\theta t)]\to0\qquad\text{in probability}.\label{eq:noisedoesntmatter}
\end{equation}
Then, we will argue that, as $t\to\infty$,
\begin{equation}
t^{-1/2}\left[h_{+}(t,-\theta t)-h_{+}(t,0)-\left(h_{-}(t,\theta t)-h_{-}(t,0)\right)\right]\Longrightarrow\mathcal{N}(0,2\theta).\label{eq:icmatters}
\end{equation}
Of course, \cref{eq:zeroval-stationary} follows immediately
from these two convergences.

Using the shear invariance \cref{eq:h_shear}, we know that
\begin{equation}
h_{+}(t,-\theta t)\overset{\mathrm{law}}{=}h_{0}(t,0)-\frac{\theta^{2}}{2}t\overset{\mathrm{law}}{=}h_{-}(t,\theta t),\label{eq:applyshearinvariance}
\end{equation}
where $h_{0}$ solves \cref{eq:KPZ} started from a two-sided
Brownian motion with zero drift. From this we conclude that 
\[
\EE[h_{+}(t,-\theta t)-h_{-}(t,\theta t)]=0.
\]
Moreover, it was shown in \cite[Theorem 1.3]{Balazs-Quastel-Seppalainen-2011}
that the variance of $h_{0}(t,0)$ is bounded by $Ct^{2/3}$ for a
constant $C<\infty$, and thus \cref{eq:applyshearinvariance}
implies that 
\[
\Var\left(h_{+}(t,-\theta t)-h_{-}(t,\theta t)\right)\le Ct^{2/3}
\]
for a new constant $C<\infty$. The limit \cref{eq:noisedoesntmatter}
then follows from Chebyshev's inequality.

Next, using the joint stationarity established in \cref{prop:joint-stationarity}
and recalling \cref{eq:DBs,eq:Sdef}, we see
that
\begin{align*}
&\left(h_{+}(t,-\theta t)-h_{+}(t,0),h_{-}(t,\theta t)-h_{-}(t,0)\right)\\&\qquad\qquad\overset{\mathrm{law}}{=}\left(B_{2}(-\theta t)-\theta^{2}t+\mathcal{S}_{\theta}(-\theta t)-\mathcal{S}_{\theta}(0),B_{1}(\theta t)-\theta^{2}t\right),
\end{align*}
and hence that 
\[
h_{+}(t,-\theta t)-h_{+}(t,0)-\left(h_{-}(t,\theta t)-h_{-}(t,0)\right)\overset{\mathrm{law}}{=}B_{2}(-\theta t)-B_{1}(\theta t)+\mathcal{S}_{\theta}(-\theta t)-\mathcal{S}_{\theta}(0).
\]
Since $\mathcal{S}_{\theta}$ is a stationary process, we see that
$t^{-1/2}[\mathcal{S}_{\theta}(\theta t)-\mathcal{S}_{\theta}(0)]$
converges to $0$ in distribution as $t\to\infty$. Then \cref{eq:icmatters}
follows from the scaling properties of Brownian motion.
\end{proof}

\subsection{Stationary case (shock reference frame)}\label{subsec:zeroflucts-shock-reference-frame}

In this section, we consider the case of initial data distributed
according to $\hat{\nu}_{\theta}$ and prove \cref{eq:shockreferenceframe-zerodiff}
of \cref{thm:btflucts}.

Although the statement of \cref{thm:btflucts}(\ref{enu:tiltedic})
is in terms of the tilted measure $\hat{\nu}_{\theta}$, we will work
with the tilt explicitly; see \cref{eq:goal-tilt} below. Therefore,
in this section we will consider $(h_{-},h_{+})(0,\cdot)\sim\nu_{\theta}$.
There are two processes $B_{1}$ and $B_{2}$, which, under $\EE$,
are standard independent two-sided Brownian motions, such that $(h_{+},h_{-})(0,\cdot)=(f_{-},f_{+})$,
with $(f_{-},f_{+})$ as in \zcref[range]{eq:f-,eq:f+}.
In particular, we have as in \cref{eq:deriv-intro} that
\begin{equation}
\frac{1}{2\theta}\partial_{x}(h_{+}-h_{-})(0,0)=\frac{1}{2\theta}\left(\int_{-\infty}^{0}\exp\left\{ B_{2}(y)-B_{1}(y)+2\theta y\right\} \,\dif y\right)^{-1}\eqqcolon R.\label{eq:Rdef}
\end{equation}
From the expression \cref{eq:Rdef}, we see that $2\theta R$
is a Gamma distributed random variable with shape $2\theta$ and rate
$1$ (see \cite{Dufresne-1990}, \cite[p. 452]{RY99}, or \cite[Theorem 6.2]{MY05}). To prove
\cref{eq:shockreferenceframe-zerodiff}, in light of \cref{eq:RNderiv} it suffices to show
that, for any $F\in\mathcal{C}_{\mathrm{b}}(\RR)$, we have
\begin{equation}
\lim_{t\to\infty}\EE\left[F\left(t^{-1/2}(h_{+}-h_{-})(t,0)\right)R\right]=\EE[F(Z)],\label{eq:goal-tilt}
\end{equation}
where $Z\sim\mathcal{N}(0,2\theta)$. If $R$ and $t^{-1/2}(h_{+}-h_{-})(t,0)$
were independent, then \cref{eq:goal-tilt} would simply be a
consequence of \cref{thm:zeroval-stationarity}. For finite $t$, the random variables
$R$ and $t^{-1/2}(h_{+}-h_{-})(t,0)$ are not independent, but the
main idea of the argument is to show that they decouple as $t\to\infty$.

Fix $\eta\in(0,\theta\wedge1)$ and $\alpha\in(0,1)$. By definition,
we have
\[
h_{\pm}(t,0)=\log\int_{\RR}Z(t,0\viiva0,y)\e^{h_{\pm}(0,y)}\,\dif y.
\]
We make the following definitions:
\begin{align}
\tilde{h}_{-}(t,0) & \coloneqq\log\int_{(-\theta-\eta)t}^{(-\theta+\eta)t}Z(t,0)\viiva0,y)\e^{B_{1}(y)-B_{1}(-t^{\alpha})+\theta y}\,\dif y;\label{eq:hminustildedef}\\
\tilde{h}_{+}(t,0) & \coloneqq\log\int_{(\theta-\eta)t}^{(\theta+\eta)t}Z(t,0\viiva0,y)\e^{B_{2}(y)+\theta y}\,\dif y;\label{eq:hplustildedef}\\
\tilde{R}_{t} & \coloneqq\left(\int_{-t^{\alpha}}^{0}\e^{B_{2}(y)-B_{1}(y)+2\theta y}\,\dif y\right)^{-1}.\label{eq:Rtildedef}
\end{align}
We note that, for sufficiently large $t$, we have $-t^{\alpha}>(-\theta+\eta)t$,
so by the independence of Brownian increments, \begin{equation}(\tilde{h}_{-}(t,0),\tilde{h}_{+}(t,0))\text{ 
is independent of }\tilde{R}_{t}.\label{eq:independence}\end{equation} 

We will show below that 
\begin{equation}
t^{-1/2}[h_{+}(t,0)-h_{-}(t,0)]-t^{-1/2}[\tilde{h}_{+}(t,0)-\tilde{h}_{-}(t,0)]\to0\label{eq:approxisgood}
\end{equation}
in probability as $t\to\infty$. First we show how this implies \cref{eq:goal-tilt}.
\begin{proof}[Proof of \cref{eq:goal-tilt} given \cref{eq:approxisgood}]
By \cref{eq:shockreferenceframe-zerodiff}, for $\iota\in\{0,1\}$,
we have
\begin{equation}
\EE\left[F(t^{-1/2}[h_{+}(t,0)-h_{-}(t,0)])R^{\iota}-F(t^{-1/2}[\tilde{h}_{+}(t,0)-\tilde{h}_{-}(t,0)])\tilde{R}_{t}^{\iota}\right]\to0\qquad\text{as }t\to\infty.\label{eq:convofEF}
\end{equation}
Here we have used the continuity and boundedness of $F$ and the fact
that $(\tilde{R}_{t})_{t\ge1}$ is uniformly integrable, which follows
from the facts that $R_{t}$ is positive and decreasing in $t$ and
$\EE[R_{1}]<\infty$ (see \cref{lem:R1moments}). But
we have by the independence \cref{eq:independence} that
\begin{align*}
\EE\left[F(t^{-1/2}[\tilde{h}_{+}(t,0)-\tilde{h}_{-}(t,0)])\tilde{R}_{t}\right] & =\EE\left[F(t^{-1/2}[\tilde{h}_{+}(t,0)-\tilde{h}_{-}(t,0)])\right]\EE\left[\tilde{R}_{t}\right].
\end{align*}
Now using \cref{eq:convofEF} with $F\equiv1$ and $\iota=1$,
we get $\EE\left[\tilde{R}_{t}\right]\to\EE\left[R\right]=1$,
and using \cref{eq:convofEF} with $\iota=0$, we get
\[
\EE\left[F(t^{-1/2}[\tilde{h}_{+}(t,0)-\tilde{h}_{-}(t,0)])\right]-\EE\left[F(t^{-1/2}[h_{+}(t,0)-h_{-}(t,0)])\right]\to0\qquad\text{as }t\to\infty.
\]
But we also know by \cref{thm:zeroval-stationarity} that
\[
\EE\left[F(t^{-1/2}[h_{+}(t,0)-h_{-}(t,0)])\right]\to\EE[F(Z)]\qquad\text{as }t\to\infty,
\]
where $Z\sim\mathcal{N}(0,2\theta)$. Combining all of these limits,
we conclude that \cref{eq:goal-tilt} holds.
\end{proof}
Now we prove \cref{eq:approxisgood}.
\begin{proof}[Proof of \cref{eq:approxisgood}.]
We recall that
\begin{align*}
h_{-}(0,x) & =B_{1}(x)-\theta x,\\
h_{+}(0,x) & =B_{2}(x)+\theta x+\mathcal{S}_{\theta}(x)-\mathcal{S}_{\theta}(0),\\
\mathcal{S}_{\theta}(x) & =\log\int_{-\infty}^{x}\e^{B_{2}(y)-B_{2}(x)-(B_{1}(y)-B_{1}(x))+2\theta(y-x)}\,\dif y.
\end{align*}
Using these facts along with the definitions \zcref[range]{eq:hminustildedef,eq:hplustildedef}, we see that
\[
\tilde{h}_{-}(t,0)=-B_{1}(-t^{-\alpha})+\log\int_{(-\theta-\eta)t}^{(-\theta+\eta)t}Z(t,0\viiva0,y)\e^{h_{-}(0,y)}\,\dif y
\]
and
\[
\tilde{h}_{+}(t,0)=\log\int_{(\theta-\eta)t}^{(\theta+\eta)t}Z(t,0\viiva0,y)\e^{h_{+}(0,y)-\mathcal{S}_{\theta}(y)+\mathcal{S}_{\theta}(0)}\,\dif y.
\]
Now since $\alpha\in(0,1)$, we have 
\begin{equation}
t^{-1/2}B_{1}(-t^{\alpha})\to0\qquad\text{as }t\to\infty\text{ in probability.}\label{eq:talphatermvanishes}
\end{equation}
We will show that
\begin{equation}
\lim_{t\to\infty}\frac{\int_{(-\theta-\eta)t}^{(-\theta+\eta)t}Z(t,0\viiva0,y)\e^{h_{-}(0,y)}\,\dif y}{\int_{\RR}Z(t,0\viiva0,y)\e^{h_{-}(0,y)}\,\dif y}=\lim_{t\to\infty}\frac{\int_{(\theta-\eta)t}^{(\theta+\eta)t}Z(t,0\viiva0,y)\e^{h_{+}(0,y)}\,\dif y}{\int_{\RR}Z(t,0\viiva0,y)\e^{h_{+}(0,y)}\,\dif y}=1\qquad\text{a.s.}\label{eq:Zsonlyneedmainregion}
\end{equation}
and that
\begin{equation}
\lim_{t\to\infty}t^{-1/2}\sup_{y\in[(\theta-\eta)t,(\theta+\eta)t]}|\mathcal{S}_{\theta}(y)|=0\qquad\text{in probability.}\label{eq:Sgoesaway}
\end{equation}
 Then \cref{eq:talphatermvanishes,eq:Zsonlyneedmainregion}
will imply that 
\[
\lim_{t\to\infty}t^{-1/2}[h_{-}(t,0)-\tilde{h}_{-}(t,0)]=0\qquad\text{in probability,}
\]
and \cref{eq:Zsonlyneedmainregion,eq:Sgoesaway}
imply that 
\[
\lim_{t\to\infty}t^{-1/2}[h_{+}(t,0)-\tilde{h}_{+}(t,0)]=0\qquad\text{in probability,}
\]
so once we prove \cref{eq:Zsonlyneedmainregion,eq:Sgoesaway}
then we can conclude \cref{eq:approxisgood} and complete the
proof.

We first turn to the proof of \cref{eq:Zsonlyneedmainregion}.
We prove the second limit, the first being analogous. Choose $\delta\in(0,\eta)$
small enough that 
\begin{equation}
0<2\sqrt{(2\theta+1)\delta-\delta^{2}/2}<\eta.\label{eq:deltaassumption}
\end{equation}
Since $h_{+}(0,\cdot)$ is a Brownian motion with drift $\theta$,
there is a random constant $C_{\delta}\in(0,\infty)$ such that
\[
(\theta-\delta)x-C_{\delta}\le h_{+}(0,x)\le(\theta+\delta)x+C_{\delta}\qquad\text{for all }x\ge0
\]
and
\[
(\theta+\delta)x-C_{\delta}\le h_{+}(0,x)\le(\theta-\delta)x+C_{\delta}\qquad\text{for all }x\le0.
\]
Fix $\eps>0$. Because we chose $\delta\in(0,\eta)$, by \cref{lem:shapefn}
and the last two displays, we may choose a random $T$ sufficiently
large that for all $t\ge T$, we have
\begin{align*}
\int_{\RR}Z(t,0\viiva0,y)\e^{h_{+}(0,y)}\,\dif y & \ge\exp\left\{ \left((\theta-\delta)^{2}/2-\frac{1}{24}-\eps\right)t-C_{\delta}\right\} ,\\
\int_{(\theta+\eta)t}^{\infty}Z(t,0\viiva0,y)\e^{h_{+}(0,y)}\,\dif y & \le\exp\left\{ \left((\theta+\delta)(\theta+\eta)-\frac{(\theta+\eta)^{2}}{2}-\frac{1}{24}+\eps\right)t+C_{\delta}\right\} ,\quad\text{and}\\
\int_{-\infty}^{(\theta-\eta)t}Z(t,0\viiva0,y)\e^{h_{+}(0,y)}\,\dif y & \le\exp\left\{ \left((\theta+\delta)(\theta-\eta)-\frac{(\theta-\eta)^{2}}{2}-\frac{1}{24}+\eps\right)t+C_{\delta}\right\} .
\end{align*}
It can quickly be checked by expanding that the right side on the
third line is less than the right side on the second line, which is
equal to
\[
\exp\left\{ \left(\frac{\theta^{2}}{2}+(\theta+\eta)\delta-\frac{\eta^{2}}{2}-\frac{1}{24}+\eps\right)t+C_{\delta}\right\} .
\]
Therefore, for $t\ge T$, we have
\begin{equation}
\frac{\int_{\RR\setminus[(\theta-\eta)t,(\theta+\eta)t]}Z(t,0)\viiva0,y)\e^{h_{+}(0,y)}\,\dif y}{\int_{\RR}Z(t,0\viiva0,y)\e^{h_{+}(0,y)}\,\dif y}\le2\exp\left\{ \left((2\theta+\eta)\delta-\frac{\eta^{2}}{2}-\frac{\delta^{2}}{2}+2\eps\right)t+2C_{\delta}\right\} .\label{eq:rat_bd}
\end{equation}
By the assumptions \cref{eq:deltaassumption} and that $\eta<1$,
we have
\[
(2\theta+\eta)\delta-\frac{\eta^{2}}{2}-\frac{\delta^{2}}{2}+2\eps<(2\theta+1)\delta-\frac{\eta^{2}}{2}-\frac{\delta^{2}}{2}+2\eps<0,
\]
as long as $\eps$ is chosen sufficiently small. Hence, for such small $\eps$, the right side of \cref{eq:rat_bd} goes to
$0$ as $t\to\infty$. This completes the proof of \cref{eq:Zsonlyneedmainregion}.

It remains to prove \cref{eq:Sgoesaway}. We start by writing
\[
\mathcal{S}_{\theta}(x)=B_{1}(x)-B_{2}(x)-2\theta x+\mathcal{Q}_{\theta}(x),
\]
with
\[
\mathcal{Q}_{\theta}(x)\coloneqq\log\int_{-\infty}^{x}\e^{B_{2}(y)-B_{1}(y)+2\theta y}\,\dif y.
\]
Now Morrey's inequality gives us, for any $p\in(1,\infty)$, a constant
$C_{p}<\infty$ such that
\[
\EE\left[\sup_{0\le y\le1}|\mathcal{Q}_{\theta}(y)|^{p}\right]\le C_{p}\left(\int_{0}^{1}\EE|\mathcal{Q}_{\theta}(y)|^{p}\,\dif y+\int_{0}^{1}\EE|\mathcal{Q}_{\theta}'(y)|^{p}\,\dif y\right),
\]
and it is easy to calculate that the right side is finite for any
$p<\infty$. Since the maximum of Brownian motion on the unit interval
also has all moments, we conclude that there is a constant $C<\infty$
such that
\[
\EE\left[\sup_{0\le y\le1}|\mathcal{S}_{\theta}(y)|^{p}\right]<C.
\]
Using the spatial stationarity of $\mathcal{S}_{\theta}$, we therefore
have
\[
\EE\left[\left(t^{-1/2}\sup_{y\in[(\theta-\eta)t,(\theta+\eta)t]}|\mathcal{S}_{\theta}(y)|\right)^{p}\right]\le Ct^{-p/2}(2\eta t+1).
\]
Choosing $p>2$, we conclude \cref{eq:Sgoesaway} by Markov's
inequality.
\end{proof}

\subsection{Flat case}

In this section, we consider the case of flat initial data and prove
\cref{thm:zeroval-flat}. The proof proceeds through several
lemmas. We make use of the following celebrated convergence of the
KPZ equation to the KPZ fixed point. To avoid unnecessary technical
details, we state the result we will use only for flat initial data,
noting that convergence is also known to hold for much more general
initial data.
\begin{prop}[\cite{KPZfixed,Wu-23}]
\label{prop:KPZ_eq_conv}Let $h$ solve \cref{eq:KPZ} started
from $h(0,\cdot)\equiv0$. Then, as $T\to\infty$, the process
\[
\left\{ 2^{1/3}T^{-1/3}\left[h(Tt,2^{1/3}T^{2/3}x)+\frac{Tt}{24}\right]\right\} _{(t,x)\in(0,\infty)\times\RR}
\]
converges in law in the topology of uniform convergence on compact
subsets of $(0,\infty)\times\RR$ to a continuous-time Markov process,
the KPZ fixed point $\{\mathfrak{h}(t,x)\}_{(t,x)\in(0,\infty)\times\RR}$
started from zero initial data. The process $\mathfrak{h}$ has continuous
sample paths.
\end{prop}
To be specific about the references for the proposition above, the KPZ fixed point $\mathfrak{h}$ was constructed in \cite{KPZfixed}. The convergence stated in \cref{prop:KPZ_eq_conv} was proved in \cite[Theorem~1.8]{Wu-23};
see also \cite{heat_and_landscape}. %
The spatial continuity of the KPZ fixed point was shown in \cite[Theorem 4.13]{KPZfixed}.

We now use \cref{prop:KPZ_eq_conv} and shear-invariance to state
the following.
\begin{lem}
\label{lem:TW_conv}For $\theta\in\RR$, suppose that $h$ solves
\cref{eq:KPZ} with initial data $h(0,x)=\theta x$. Then we
have the distributional convergence
\begin{equation}
\frac{h(t,0)+\left(\frac{1}{24}-\frac{\theta^{2}}{2}\right)t}{t^{1/3}}\Longrightarrow\frac{X}{2},\label{eq:TWconvergence}
\end{equation}
where $X$ is a Tracy--Widom GOE random variable.
\end{lem}

\begin{proof}
For $\theta=0$, \cref{prop:KPZ_eq_conv} implies the convergence
as $t\to\infty$ of the rescaled process
\[
x\mapsto t^{-1/3}\left[h(t,2^{1/3}t^{2/3}x)+\frac{t}{24}\right]
\]
to the KPZ fixed point with initial data at time $1$, denoted $x\mapsto2^{-1/3}\mathfrak{h}(1,x)$,
in the sense of uniform convergence on compact sets. By \cite[(4.15)]{KPZfixed},
the process $x\mapsto2^{-1/3}\mathfrak{h}(1,x)$ has the law of $x\mapsto\mathcal{A}_{1}(2^{-2/3}x)$,
where $\mathcal{A}_{1}$ is the $\mathrm{Airy}_{1}$ process. And
it is known \cite{Ferrari-Spohn-2005,Sasamoto-2005} (see also \cite{ReflectedBM_book})
that $\mathcal{A}_{1}$ is a stationary process whose marginals are
distributed according to $1/2$ times the Tracy--Widom GOE distribution.
This implies the convergence \cref{eq:TWconvergence} in the
case $\theta=0$.

The case $\theta\ne0$ then follows from the shear-invariance of the
KPZ equation and the stationarity of the increments of $x\mapsto h(t,x)$
given the flat initial condition $\theta x$. To be precise, the shear
invariance stated in \cref{eq:h_shear} implies that $h(t,0)$
has the same distribution as $h_{0}(t,\theta t)+\frac{\theta^{2}}{2}t$,
where $h_{0}$ solves \cref{eq:KPZ} started from $0$ initial
data. Also, the shift invariance of $Z$ stated in \cref{eq:Z_shift}
implies that
\begin{equation} \label{eq:flat_shift}
\begin{aligned}
h_{0}(t,\theta t)&=\log\int_{\RR}Z(t,\theta t\viiva0,y)\,\dif y\\\overset{\mathrm{law}}&{=}\log\int_{\RR}Z(t,0\viiva0,y-\theta t)\,\dif y=\log\int_{\RR}Z(t,0\viiva0,y)\,\dif y=h_{0}(t,0).
\end{aligned}
\end{equation}
Therefore, we have $h(t,0)-\frac{\theta^{2}}{2}t\overset{\mathrm{law}}{=}h_{0}(t,0)$.
Using this identity in law, the convergence \cref{eq:TWconvergence}
in the general $\theta$ case follows from the $\theta=0$ case.
\end{proof}
\cref{lem:TW_conv} will be the ingredient yielding the Tracy--Widom
GOE random variables in claimed in \cref{thm:zeroval-flat}.
To complete the proof of \cref{thm:zeroval-flat}, we also need
to know that the Tracy--Widom GOE random variables coming from $h_{+}(t,0)$
and $h_{-}(t,0)$ are independent. That is the task of the rest of
this section. 

The idea of the proof of independence is that, due to the shear-invariance
\cref{eq:h_shear}, the contribution of the space-time white noise noise to $h_{+}(t,0)$ 
mostly comes from the right of the $t$-axis, while the contribution
to $h_{-}(t,0)$ mostly comes from the left of the $t$-axis (here we represent the time $t$ as the vertical coordinate). The
fact that the noises in these regions are independent yields the independence
of the limits.

To carry out this argument, we use the continuum directed random polymer
constructed in \cite{Alberts-Khanin-Quastel-2014a}, as well as estimates
on the behavior of this polymer proved in \cite{Das-Zhu-22b}. For
$t>0$ and $\theta,x\in\RR$, we let $Q_{\theta,t,x}$ denote the
random measure of the point-to-line polymer, with mean slope $-\theta$,
from $\{0\}\times\RR$ to $(t,x)$. If $Y\in\mathcal{C}([0,T])$ denotes
the random polymer path, this means that, for any $0\le t_{1}\le\cdots\le t_{n}\le t$
and $y_{1},\ldots,y_{n}\in\RR$, we have
\begin{equation}
Q_{\theta,t,x}(Y_{1}\in\dif x_{1},\ldots,Y_{n}\in\dif x_{n})=\frac{\int_{\RR}\e^{\theta x_{0}}\prod\limits_{j=0}^{n}Z(t_{j+1},x_{j+1}\viiva t_{j},x_{j})\,\dif x_{0}}{\int_{\RR}\e^{\theta y}Z(t,x\viiva0,y)\,\dif y}\dif x_{1}\cdots\dif x_{n},\label{eq:Qdef}
\end{equation}
where $t_{0}=0$, $t_{n+1}=t$, and $x_{n+1}=x$. 

Now define 
\[
A_{t,\pm}=\{Y\in\mathcal{C}([0,t])\st\pm Y(s)>0\text{ for all }s\in[0,t\}\}.
\]
Our first lemma says that if a polymer starts at distance $t^{1/2}$
to the right of the origin and has a positive drift, then it is unlikely
to ever cross to the left of the origin. Note that the typical annealed displacement
of the polymer is on the order $t^{2/3}$, so the positive drift is
really required for this statement to be true. The power $1/2$ is
rather arbitrary; the lemma holds with any power strictly greater
than $1/3$, but we want the initial displacement to be $o(t^{2/3})$
so that the value of $h$ is close to $h(t,0)$, as shown in \cref{eq:fixendpoint}
below.
\begin{lem}
\label{lem:QAlimit}If $\theta>0$ is fixed, then
\begin{equation}
\lim_{t\to\infty}Q_{\pm\theta,t,\pm t^{1/2}}(A_{t,\pm})=1\qquad\text{in probability.}\label{eq:QAlimit}
\end{equation}
\end{lem}

\begin{proof}
We prove the $+$ case, as the $-$ case is symmetrical. We note that
\begin{equation}
Q_{\theta,t,x}(\dif Y)\overset{\mathrm{law}}{=}\overleftarrow{Q}_{t}(\dif(s\mapsto Y(t-s)-x+\theta s)),\label{eq:revtime}
\end{equation}
where $\overleftarrow{Q}_{t}$ is the random measure of a point-to-line
continuum directed random polymer from $(0,0)$ to $\{t\}\times\RR$,
without drift. If we set $Y(s)=X(t-s)+x+\theta(t-s)$, then we have
for any $s\in[0,t]$ and $x\ge0$ that
\begin{equation}
Y(t-s)\ge0\iff X(s)\ge-x-\theta s\impliedby|X(s)|\le x+\theta s.\label{eq:YtoX}
\end{equation}

Now we apply \cite[Proposition 3.3-(point-to-line)]{Das-Zhu-22b},
with $\eps\leftarrow t^{-1}$ and $t\leftarrow0$, to obtain, for
every $\delta\in(0,1/2)$, constants $C_{1},C_{2}<\infty$ depending
only on $\delta$ such that, for all $m\ge1$, we have
\begin{equation}
\mathbb{P}\left(\overleftarrow{Q}_{t}\left(\sup_{s\in[0,t]}\frac{|X(s)|}{t^{1/6+\delta}s^{1/2-\delta}}\ge m\right)\ge C_{1}\e^{-m^{2}/C_{1}}\right)\le C_{2}\e^{-m^{3}/C_{2}},\label{eq:DZresult}
\end{equation}
where we use $X$ for the continuum directed random polymer under
the measure $\overleftarrow{Q}_{t}$. By Young's inequality, we have
a constant $C_{3}<\infty$, depending only on $\delta$ and $\theta$,
such that
\begin{equation}
mt^{1/6+\delta}s^{1/2-\delta}\le C_{3}(mt^{1/6+\delta})^{2/(1+2\delta)}+\theta s.\label{eq:useyoung}
\end{equation}
Taking $m=t^{1/12-\delta/2}/C_{3}^{1/2+\delta}$ and $\delta=1/12$,
the right side of \cref{eq:useyoung} becomes $t^{1/2}+\theta s$,
and then from \cref{eq:DZresult} we obtain constants $C_{4},C_{5}<\infty$,
depending only on $\theta$, such that for sufficiently large $t$,
we have
\[
\mathbb{P}\left(\overleftarrow{Q}_{t}\left(\exists s\in[0,t]\text{ s.t. }|X(s)|\ge t^{1/2}+\theta s\right)\ge C_{4}\e^{-t^{1/2}/C_{4}}\right)\le C_{5}\e^{-t^{1/8}/C_{5}}.
\]
Using this along with \cref{eq:revtime,eq:YtoX},
we obtain
\[
\mathbb{P}\left(Q_{\theta,t,t^{1/2}}(A_{t,+})\le1-C_{4}\e^{-t^{1/12}/C_{4}}\right)\le C_{5}e^{-t^{1/8}/C_{5}},
\]
which implies \cref{eq:QAlimit}.
\end{proof}
\begin{lem}
\label{lem:hlSHE}If we define
\begin{equation}
\phi_{\pm}(t,x)=Q_{\pm\theta,t,x}(A_{t,\pm})\int_{\RR}Z(t,x\viiva0,y)\e^{\pm\theta y}\,\dif y,\label{eq:phidef}
\end{equation}
then $\phi_{\pm}$ solves the half-line stochastic heat equation
\begin{align}
\dif\phi_{\pm}(t,x) & =\frac{1}{2}\Delta\phi_{\pm}(t,x)\dif t+\phi_{\pm}(t,x)\dif W(t,x), & t,\pm x>0;\label{eq:hlSHE}\\
\phi_{\pm}(0,x) & =\e^{\pm\theta x}, & \pm x>0;\label{eq:hlSHE-ic}\\
\phi_{\pm}(t,0) & =0, & t>0.\label{eq:hlSHE-bc}
\end{align}
\end{lem}
The quantities $\phi_\pm$ considered in \cref{lem:hlSHE} are the partition functions of the continuum directed random polymer but with the integral only taken over paths that stay in the respective half-space.
Before we prove \cref{lem:hlSHE}, we state the following corollary,
which is clear from \cref{lem:hlSHE} and the well-posedness
of the stochastic heat equation on the half-line.
\begin{cor}
\label{cor:phipmmeasurable}The process $(\phi_{+}(t,x))_{t,x\ge0}$
is measurable with respect to the restriction of $\dif W$ to $[0,\infty)^{2}$,
and the process $(\phi_{-}(t,x))_{t\ge0,x\le0}$ is measurable with
respect to the restriction of $\dif W$ to $[0,\infty)\times(-\infty,0]$,
and hence these two processes are independent of each other.
\end{cor}

We will prove the $+$ case of \cref{lem:hlSHE}; the $-$ case
is symmetrical. We use an approximation argument. For $\eps>0$, define
\[
A_{t,+}^{(\eps)}\coloneqq\{Y(s)>0\text{ for all }s\in[0,t]\cap\eps\ZZ\}
\]
and
\[
\phi_{+}^{(\eps)}(t,x)=Q_{\theta,t,x}(A_{t,+}^{(\eps)})\int_{\RR}Z(t,x\viiva0,y)\e^{\theta y}\,\dif y.
\]
For $t,x>0$, we have by \cref{eq:Qdef} that 
\begin{equation}
\phi_{+}^{(\eps)}(t,x)=\int_{(0,\infty)^{J+1}}\e^{\theta x_{0}}\prod_{j=0}^{J}Z(\mathfrak{t}_{j+1},x_{j+1}\viiva\mathfrak{t}_{j},x_{j})\,\dif x_{0}\cdots\dif x_{J},\label{eq:phiepsdef}
\end{equation}
where we have defined
\[
\mathfrak{t}_{0}=0,\qquad(\mathfrak{t}_{J+1},x_{J+1})=(t,x),\qquad\{\mathfrak{t}_{1}\le\cdots\le\mathfrak{t}_{J}\}=(0,t)\cap\eps\ZZ.
\]
 Now, recalling that $G$ is the standard heat kernel, we also define,
for $s<t$ and $x,y\in\RR$,
\begin{equation}
G_{+}^{(\eps)}(t,x\viiva s,y)=\int_{(0,\infty)^{J}}\prod_{k=0}^{K}G(\mathfrak{s}_{k+1}-\mathfrak{s}_{k},x_{k+1}-x_{k})\,\dif x_{1}\cdots\dif x_{K},\label{eq:Gepsdef}
\end{equation}
where, here, we use the notation
\[
(\mathfrak{s}_{0},x_{0})=(s,y),\qquad(\mathfrak{s}_{K+1},x_{K+1})=(t,x),\qquad\{\mathfrak{s}_{1}\le\cdots\le\mathfrak{s}_{K}\}=(s,t)\cap\eps\ZZ.
\]
We also note that
\begin{equation}
\frac{G_{+}^{(\eps)}(t,x\viiva s,y)}{G(t-s,x-y)}=\mathbb{P}_{t,x\viiva s,y}\left(X_{\mathfrak{s}_{k}}\ge0\text{ for all }k\in\{1,\ldots,K\}\right),\label{eq:BBrep}
\end{equation}
where $\mathbb{P}_{t,x\viiva s,y}$ is the probability measure under
which $X$ is Brownian bridge with unit quadratic variation such that
$X_{s}=y$ and $X_{t}=x$. We note here that $K$ depends on $\eps$. It follows from \cref{eq:BBrep} and
the continuity of Brownian bridge that, for each $t,s,x,y$, we have
\begin{equation}
    \begin{aligned}
        \lim_{\eps\downarrow0}G_{+}^{(\eps)}(t,x\viiva s,y)&=G(t-s,x-y)\mathbb{P}_{t,x\viiva s,y}\left(X_{r}\ge0\text{ for all }r\in(s,t)\right)\\&=[G(t-s,x-y)-G(t-s,x+y)]\mathbf{1}\{x,y\ge0\}\eqqcolon G_+(t,x\viiva s,y).
\end{aligned}\label{eq:Gepslimit}
\end{equation}
The second identity is the formula for the transition density of Brownian motion killed at the origin, which is obtained by the reflection principle; see e.g.~\cite[(2.8.9)]{Karatzas-Shreve-88}.

It also follows from the definition \cref{eq:Gepsdef} that,
if $k\in\ZZ$, $s\le k\eps\le t$, and $x,y\in\RR$, then
\begin{equation}
\int_{(0,\infty)}G_{+}^{(\eps)}(t,x\viiva k\eps,z)G_{+}^{(\eps)}(k\eps,z\viiva s,y)\,\dif z=G_{+}^{(\eps)}(t,x\viiva s,y).\label{eq:Geps+KC}
\end{equation}
The following lemma says that the approximation $\phi^{(\eps)}$ solves
an approximation of the mild solution formula for \cref{eq:hlSHE}--\cref{eq:hlSHE-bc}:
\begin{lem}
We have
\begin{equation}
\phi_{+}^{(\eps)}(t,x)=\int_{\RR}G_{+}^{(\eps)}(t,x\viiva0,y)\e^{\theta y}\,\dif y+\int_{0}^{t}\int_{\RR}G_{+}^{(\eps)}(t,x\viiva s,y)\phi_{+}^{(\eps)}(s,y)\,\dif W(s,y).\label{eq:phiepsmild}
\end{equation}
\end{lem}

\begin{proof}
We proceed by induction on $t$. First suppose that $t\in[0,\eps]$.
Then we have $J=0$ and
\[
\phi_{+}^{(\eps)}(t,x)=\int_{(0,\infty)}\e^{\theta x_{0}}Z(t,x\viiva0,x_{0})\,\dif x_{0}.
\]
Also, for all $s\in[0,t]$, we have in this case
\[
G_{+}^{(\eps)}(t,x\viiva s,y)=G(t-s,x-y).
\]
Thus, in this case, \cref{eq:phiepsmild} is simply the mild
solution formula for the stochastic heat equation.

Now suppose that \cref{eq:phiepsmild} for all $t\le k\eps$.
We will use this inductive hypothesis to prove \cref{eq:phiepsmild}
for $t\in(k\eps,(k+1)\eps]$. So let $t\in(k\eps,(k+1)\eps]$. Since
$G_{+}^{(\eps)}(s',x\viiva s,y)=G(s'-s,x-y)$ for all $k\eps<s\le s'\le(k+1)\eps$,
and $\phi^{(\eps)}$ satisfies the stochastic heat equation on $(k\eps,(k+1)\eps]$
with initial condition $\phi^{(\eps)}(k\eps,x)=\phi^{(\eps)}(t,x)\mathbf{1}\{x\ge0\}$,
the mild solution formula for the stochastic heat equation again tells
us that
\[
\phi_{+}^{(\eps)}(t,x)=\int_{(0,\infty)}G_{+}^{(\eps)}(t,x\viiva k\eps,y)\phi_{+}^{(\eps)}(k\eps,y)\,\dif y+\int_{k\eps}^{t}\int_{\RR}G_{+}^{(\eps)}(t,x\viiva s,y)\phi_{+}^{(\eps)}(s,y)\,\dif W(s,y).
\]
By the inductive hypothesis, we have
\begin{align*}
\int_{(0,\infty)} & G_{+}^{(\eps)}(t,x\viiva k\eps,y)\phi_{+}^{(\eps)}(k\eps,y)\,\dif y\\
 & =\int_{(0,\infty)}G_{+}^{(\eps)}(t,x\viiva k\eps,y)\int_{\RR}G_{+}^{(\eps)}(k\eps,y\viiva0,y')\e^{\theta y'}\,\dif y'\,\dif y\\
 & \qquad+\int_{(0,\infty)}G_{+}^{(\eps)}(t,x\viiva k\eps,y)\left(\int_{0}^{t}\int_{\RR}G_{+}^{(\eps)}(k\eps,y\viiva s,y')\phi_{+}^{(\eps)}(s,y')\,\dif W(s,y')\right)\,\dif y\\
 & =\int_{\RR}\e^{\theta y}G_{+}^{(\eps)}(t,x\viiva0,y)\,\dif y+\int_{0}^{t}\int_{\RR}G_{+}^{(\eps)}(t,x\viiva s,y)\phi_{+}^{(\eps)}(s,y)\,\dif W(s,y),
\end{align*}
where in the last identity we used \cref{eq:Geps+KC} on each
term. This completes the inductive step and hence the proof.
\end{proof}
Now we can complete the proof of \cref{lem:hlSHE}.
\begin{proof}[Proof of \cref{lem:hlSHE}]
The sequence of sets $(A_{t,+}^{2^{-n}})_{n\in\NN}$ is decreasing,
and the continuity of $Y$ implies that $\bigcap\limits_{n\in\NN}A_{t,+}^{2^{-n}}=A_{t,+}$.
By the definitions \cref{eq:phidef,eq:phiepsdef},
this means that the sequence $(\phi_{+}^{2^{-n}})_{n\in\NN}$ is almost-surely
decreasing in $n$ and that, for each $t,x$, we have
\begin{equation}
\lim_{n\to\infty}\phi_{+}^{2^{-n}}(t,x)=\phi_{+}(t,x)\qquad\text{a.s.}\label{eq:limphi}
\end{equation}
Using \eqref{eq:BBrep}, we have
\begin{equation}
0\le G_{+}^{(2^{-n})}(t,x\viiva0,y)\le G(t-s,x-y)\mathbf{1}\{y\ge0\},\label{eq:Gestimate}
\end{equation}
and thus, we have by \cref{eq:Gepslimit} and the dominated convergence
theorem that
\begin{equation}
\lim_{n\to\infty}\int_{\RR}G_{+}^{(2^{-n})}(t,x\viiva0,y)\e^{\theta y}\,\dif y=\int_{(0,\infty)}G_+(t,x\viiva s,y)\e^{\theta y}\,\dif y.\label{eq:dettermconverges}
\end{equation}
Moreover, we have by the Itô isometry that
\begin{align}
 & \EE\left|\int_{0}^{t}\int_{\RR}\left(G_{+}^{(2^{-n})}(t,x\viiva s,y)\phi_{+}^{(2^{-n})}(s,y)-G_+(t,x\viiva s,y)\phi_{+}(s,y)\right)\,\dif W(s,y)\right|^{2}\nonumber \\
 & \ =\int_{0}^{t}\int_{\RR}\EE\left|G_{+}^{(\eps)}(t,x\viiva s,y)\phi_{+}^{(2^{-n})}(s,y)-G_+(t,x\viiva s,y)\phi_{+}(s,y)\right|^{2}\,\dif y\,\dif s.\label{eq:secondmomentintegrals}
\end{align}
Standard moment estimates for the stochastic heat equation on the line (see e.g.~\cite{Khoshnevisan-2014}), along with the fact that $\phi_+^{(2^{-n})}(s,y)$ is decreasing in $n$, show that
\[
\EE|\phi_{+}^{(2^{-n})}(s,y)|^{4}\vee \EE|\phi_{+}(s,y)|^{4}\le C\e^{4\theta y}
\]
for some constant $C<\infty$ independent of $n$. This and \cref{eq:Gestimate}
allow us to use uniform integrability and the dominated convergence
theorem with \cref{eq:Gepslimit,eq:limphi} in \cref{eq:secondmomentintegrals} to see that 
\begin{equation}
\begin{aligned}\lim_{n\to\infty} & \int_{0}^{t}\int_{\RR}G_{+}^{(2^{-n})}(t,x\viiva s,y)\phi_{+}^{(2^{-n})}(s,y)\,\dif W(s,y)\\
 & =\int_{0}^{t}\int_{\RR}[G(t-s,x-y)-G(t-s,x+y)]\mathbf{1}\{y\ge0\}\phi_{+}(s,y)\,\dif W(s,y)
\end{aligned}
\label{eq:stochtermconverges}
\end{equation}
in probability. Now using \cref{eq:limphi,eq:dettermconverges,eq:stochtermconverges} in \cref{eq:phiepsmild},
we see that 
\[
\phi_{+}(t,x)=\int_{0}^{\infty}G_+(t,x\viiva 0,y)\e^{\theta y}\,\dif y+\int_{0}^{t}\int_{0}^{\infty}G_+(t,x\viiva s,y)\phi_{+}(s,y)\,\dif W(s,y),
\]
and hence that $\phi_{+}$ is a mild solution to \cref{eq:hlSHE}--\cref{eq:hlSHE-bc}.
\end{proof}
Now we can complete the proof of \cref{thm:zeroval-flat}.
\begin{proof}[Proof of \cref{thm:zeroval-flat}]
We note that
\begin{equation}
h_{-}(t,x)=\log\int_{\RR}Z(t,x\viiva0,y)\e^{-\theta y}\,\dif y\qquad\text{and}\qquad h_{+}(t,x)=\log\int_{\RR}Z(t,x\viiva0,y)\e^{\theta y}\,\dif y.\label{eq:hplushminus}
\end{equation}
Comparing \cref{eq:hplushminus} with \cref{eq:phidef}
and using \cref{lem:QAlimit}, we see that
\begin{equation}
|h_{-}(t,-t^{1/2})-\log\phi_{-}(t,-t^{1/2})|+|h_{+}(t,t^{1/2})-\log\phi_{+}(t,t^{1/2})|\xrightarrow[t\to\infty]{\mathrm{P}}0.\label{eq:convergetophipm}
\end{equation}
By the shear invariance in \cref{rmk:shear} and the shift invariance in \eqref{eq:flat_shift}, we have that $h_-(t,0) - h_-(t,- t^{-1/2})$ and $h_+(t,0) - h_+(t, t^{-1/2})$ (marginally) have the laws of
\[
h_0(t,0) - h_0(t,-t^{1/2}) - \theta t^{1/2},\quad\text{and}\quad h_0(t,0) - h_0(t,t^{1/2}) - \theta t^{1/2}, 
\]
respectively, where $h_0$ solves the KPZ equation with $0$ (flat) initial data. Also, by \cref{prop:KPZ_eq_conv},
the rescaled KPZ equation with flat initial data converges to the KPZ fixed point (which has continuous sample paths), in the sense
of uniform convergence on compact sets. 
 Using this and noting that $t^{1/2}=o(t^{2/3})$ as $t\to\infty$, we have
\begin{equation}
t^{-1/3}|h_{-}(t,0)-h_{-}(t,-t^{1/2}) + \theta t^{1/2}|+t^{-1/3}|h_{+}(t,0)-h_{+}(t,t^{1/2}) + \theta t^{1/2}|\xrightarrow[t\to\infty]{\mathrm{P}}0.\label{eq:fixendpoint}
\end{equation}
Next, by \cref{lem:TW_conv}, we know that
\[
\frac{h_{-}(t,0)+\left(\frac{1}{24}-\frac{\theta^{2}}{2}\right)t}{t^{1/3}}\qquad\text{and}\qquad\frac{h_{+}(t,0)+\left(\frac{1}{24}-\frac{\theta^{2}}{2}\right)t}{t^{1/3}}
\]
each converge in distribution, as $t\to\infty$, to $X/2$, where
$X$ is a Tracy--Widom GOE random variable. Combining this observation
with \cref{eq:convergetophipm,eq:fixendpoint},
we see that
\[
t^{-1/3}\left[\log\phi_{+}(t,t^{1/2})+\left(\frac{1}{24}-\frac{\theta^{2}}{2}\right)t - \theta t^{1/2} \right],\quad\text{and}\quad t^{-1/3}\left[\log\phi_{-}(t,-t^{1/2})+\left(\frac{1}{24}-\frac{\theta^{2}}{2}\right)t - \theta t^{1/2}\right]
\]
each converge in distribution to $X/2$. On the other hand, \cref{cor:phipmmeasurable}
tells us that $\phi_{-}(t,-t^{1/2})$ and $\phi_{+}(t,t^{1/2})$ are
independent of one another, and so in fact we have
\[
\left(t^{-1/3}\left[\log\phi_{-}(t,-t^{1/2})+\left(\frac{1}{24}-\frac{\theta^{2}}{2}\right)t - \theta t^{1/2}\right],t^{-1/3}\left[\log\phi_{+}(t,t^{1/2})+\left(\frac{1}{24}-\frac{\theta^{2}}{2}\right)t - \theta t^{1/2}\right]\right)
\]
converges to $\left(\frac{X_{1}}{2},\frac{X_{2}}{2}\right)$, where $X_{1}$ and $X_{2}$ are independent Tracy--Widom GOE random
variables. Subtracting, and then applying \cref{eq:convergetophipm,eq:fixendpoint} again, yields \cref{eq:GOEconv}.
\end{proof}

\section{V-shaped solutions}\label{sec:V-shaped}

In this section we complete the proofs of \cref{thm:no-v-shaped,thm:whatcanyouconvergeto,thm:asconvergence}. The idea of the proof is
to write a V-shaped solution using the $V$ function from two solutions
with asymptotic slopes (\cref{lem:A_map}). The convergence of
these asymptotically sloped solutions described in \cref{prop:spatially-homog-convergence}
will then let us consider these sloped solutions as stationary.
\begin{lem}
\label{lem:A_map}Let $\theta>0$. There is a (deterministic) measurable
function $\mathbf{A}\colon\mathcal{V}(\theta)\to\mathcal{X}(\theta)$
such that, for all $f_{\mathsf{V}}\in\mathcal{V}(\theta)$, we have
\begin{equation}
V[\mathbf{A}[f_{\mathsf{V}}]]=f_{\mathsf{V}}.\label{eq:VAid}
\end{equation}
 
\end{lem}

\begin{proof}
  We define
  \[
    \mathbf{A}[f_{\mathsf{V}}]\coloneqq \left(f_{\mathsf{V}}(x) - \log\frac{\e^{2\theta x}+1}{2},f_{\mathsf{V}}(x) - \log\frac{\e^{-2\theta x}+1}{2}\right).
  \]
Then, recalling the definition \cref{eq:Vdef}, we have
\[
  V[\mathbf{A}[f_{\mathsf{V}}]](x) =   \log\frac{\e^{f_{\mathsf{V}}(x)}\cdot\frac{2}{\e^{2\theta x}+1} +\e^{f_{\mathsf{V}}(x)}\cdot\frac{2}{\e^{-2\theta x}+1}}{2} = f_{\mathsf{V}}(x).
\]
From the definitions \cref{eq:Vthetadef,eq:Ythetadef}, we see that if $f_{\mathsf{V}}\in \mathcal{V}(\theta)$ then $\mathbf{A}[f_{\mathsf{V}}] \in \mathcal{Y}(\theta)$.
Moreover, if $(f_-,f_+) = \mathbf{A}[f_{\mathsf{V}}]$, then $f_+(x)-f_-(x) = 2\theta x$, which is evidently increasing in $x$, so indeed (recalling the definition \cref{eq:Xthetadef}) we have $\mathbf{A}[f_{\mathsf{V}}]\in\mathcal{X}(\theta)$ as well.
\end{proof}
\begin{lem}
\label{lem:Vclosetooneofthem}Let $\theta>0$ and suppose that $(f_{-},f_{+})\in\mathcal{X}(\theta)$.
Let $f_{\mathsf{V}}=V[f_{-},f_{+}]$.
\begin{enumerate}
\item \label{enu:fplusgtr}If $f_{+}(0)\ge f_{-}(0)$, then for all $x\ge0$,
we have $|f_{\mathsf{V}}(x)-f_{\mathsf{V}}(0)-(f_{+}(x)-f_{+}(0))|\le\log2$.
\item \label{enu:fminusgtr}If $f_{+}(0)\le f_{-}(0)$, then for all $x\le0$,
we have $|f_{\mathsf{V}}(x)-f_{\mathsf{V}}(0)-(f_{-}(x)-f_{-}(0))|\le\log2$.
\end{enumerate}
\end{lem}

\begin{proof}
We prove the first assertion; the proof of the second is similar.
So suppose that
$f_{+}(0)\ge f_{-}(0)$. Then we have, for all $x\ge0$, that
\[
f_{\mathsf{V}}(x)-f_{\mathsf{V}}(0)=\log\frac{\e^{f_{+}(x)}+\e^{f_{-}(x)}}{\e^{f_{+}(0)}+\e^{f_{-}(0)}}\ge\log\frac{\e^{f_{+}(x)}}{2\e^{f_{+}(0)}}=f_{+}(x)-f_{+}(0)-\log2.
\]
On the other hand, we have
\begin{align*}
f_{\mathsf{V}}(x)-f_{\mathsf{V}}(0)=\log\frac{\e^{f_{+}(x)}(1+\e^{f_{-}(x)-f_{+}(x)})}{\e^{f_{+}(0)}+\e^{f_{-}(0)}} & \le f_{+}(x)-f_{+}(0)+\log(1+\e^{f_{-}(x)-f_{+}(x)})\\
 & \le f_{+}(x)-f_{+}(0)+\log2,
\end{align*}
where in the last inequality we used that $f_-(x)-f_+(x)\le f_-(0)-f_+(0)\le 0$ since we assumed that $(f_-,f_+)\in\mathcal{X}(\theta)$. This completes the proof.
\end{proof}
The following proposition is a more precisely stated version of \cref{thm:no-v-shaped}.
\begin{prop}
\label{prop:noVshaped}There is no probability measure $\nu_{\mathsf{V}}$ on $\mathcal{C}_{\mathrm{KPZ};0}$ such that $\nu_{\mathsf{V}}(\mathcal{V}(\theta)\cap\mathcal{C}_{\mathrm{KPZ};0})=1$ and such that, if $h_{\mathsf{V}}$ solves \cref{eq:KPZ}
with initial data $h_{\mathsf{V}}(0,\cdot)\sim\nu_{\mathsf{V}}$ (independent
of the noise), then
\[
h_{\mathsf{V}}(t,\cdot)-h_{\mathsf{V}}(t,0)\sim\nu_{\mathsf{V}}\qquad\text{for all }t\ge0.
\]
\end{prop}

\begin{proof}
Suppose for the sake of contradiction that there does exist such a
measure $\nu_{\mathsf{V}}$. Define $\mathbf{h}=(h_{-},h_{+})$, and let %
$h_{-},h_{+},h_{\mathsf{V}}$ each solve \cref{eq:KPZ},
with initial conditions $h_{\mathsf{V}}\sim\nu_{\mathsf{V}}$ and
$\mathbf{h}(0,\cdot)=\mathbf{A}[h_{\mathsf{V}}(0,\cdot)]$. Here,
$\mathbf{A}$ is defined as in \cref{lem:A_map}. 
Recalling \cref{eq:VAid}, this means that $V[\mathbf{h}(0,\cdot)]=h_{\mathsf{V}}(0,\cdot)$,
and hence by \cref{prop:Vsolves} we in fact have
\begin{equation}
V[\mathbf{h}(t,\cdot)](x)=h_{\mathsf{V}}(t,x)\qquad\text{for all }t\ge0\text{ and }x\in\RR.\label{eq:Videntity}
\end{equation}
Let $U_{T}\sim\Uniform([0,T])$ be independent of everything else.
By \cref{prop:spatially-homog-convergence}, we have
\begin{equation}
\Law(\pi_{0}[\mathbf{h}(U_{T},\cdot)])\to\nu_{\theta}\qquad\text{weakly w.r.t. the topology of }\mathcal{C}_{\mathrm{KPZ};0}^{2}\text{ as }T\to\infty.\label{eq:hplushminusconverge}
\end{equation}
We now show that
\begin{equation} \label{eq:hUt_tight}
\text{for all }\eps > 0, \text{ there exists } K < \infty \text{ such that }\sup_{T\in(0,\infty)}\mathbb{P}\left(h_{+}(U_{T},0)-h_{-}(U_{T},0)>K\right)<\eps.
\end{equation}
Suppose for the sake of contradiction that there is some $\eps>0$
and a sequence $T_{k}\uparrow\infty$ such that
\begin{equation}
\inf_{k\in\NN}\mathbb{P}\left((h_{+}-h_{-})(U_{T_{k}},0)>k\right)\ge\eps.\label{eq:gtrk}
\end{equation}
Recalling \cref{eq:Videntity} and the definition \cref{eq:Vdef}
of $V$, we have
\begin{align}
  h_{\mathsf{V}}&(t,x)-h_{\mathsf{V}}(t,0)-(h_{+}(t,x)-h_{+}(t,0))\notag\\&=\log\frac{\e^{h_{+}(t,x)}+\e^{h_{-}(t,x)}}{2}-\log\frac{\e^{h_{+}(t,0)}+\e^{h_{-}(t,0)}}{2}-\left(\log\e^{h_{+}(t,x)}-\log\e^{h_{+}(t,0)}\right)\notag\\&=\log\frac{\left(\e^{h_{+}(t,x)}+\e^{h_{-}(t,x)}\right)\e^{-h_{+}(t,x)}}{\left(\e^{h_{+}(t,0)}+\e^{h_{-}(t,0)}\right)\e^{-h_{+}(t,0)}}=\log\frac{\e^{-(h_{+}-h_{-})(t,x)}+1}{\e^{-(h_{+}-h_{-})(t,0)}+1}\notag\\&=\log\frac{\e^{-\left((h_{+}-h_{-})(t,x)-(h_{+}-h_{-})(t,0)\right)}+\e^{(h_{+}-h_{-})(t,0)}}{1+\e^{(h_{+}-h_{-})(t,0)}}.\label{eq:diff}
\end{align}
For any $x\in(-\infty,0)$, the tightness implied by \cref{eq:hplushminusconverge}
means that there is some $A(x) \in (0,\infty)$ such that
\begin{equation}
\sup_{k\in\NN}\mathbb{P}\left(\left|(h_{+}-h_{-})(U_{T_{k}},x)-(h_{+}-h_{-})(U_{T_{k}},0)\right|>A(x)\right)<\frac{\eps}{4}.\label{eq:ltA}
\end{equation}
Also, by the convergence $\Law(\pi_{0}[\mathbf{h}(t,\cdot)])\to\nu_{\theta}$
in \cref{eq:hplushminusconverge} and since the second marginal
of $\nu_{\theta}$ is a Brownian motion with drift $\theta$, there
is an $M_{0}<\infty$ such that, if $x<-M_{0}$, then there is a $C(x)<\infty$
such that
\begin{equation}
\sup_{k\ge C(x)}\mathbb{P}\left(h_{+}(U_{T_{k}},x)-h_{+}(U_{T_{k}},0)>\frac{1}{2}\theta x\right)<\frac{\eps}{4}.\label{eq:ltthetax}
\end{equation}
Now combining \cref{eq:gtrk,eq:ltA,eq:ltthetax},
we see that for all $x\le-M_{0}$ and $k\ge C(x)$, with probability at least $\eps/2$, we have 
\[\begin{array}{c}(h_{+}-h_{-})(U_{T_{k}},0)>k,\qquad h_{+}(U_{T_{k}},x)-h_{+}(U_{T_{k}},0)\le\frac{1}{2}\theta x,\qquad\text{and}\\
\grave{\left|(h_{+}-h_{-})(U_{T_{k}},x)-(h_{+}-h_{-})(U_{T_{k}},0)\right|}\le A(x).\end{array}\] This means that with probability
at least $\eps/2$ we have
\begin{align*}
    h_{\mathsf{V}}(U_{T_{k}},x)-h_{\mathsf{V}}(U_{T_{k}},0) \overset{\cref{eq:diff}}&{=}h_{+}(U_{T_{k}},x)-h_{+}(U_{T_{k}},0)\\&\qquad+\log\frac{\e^{-\left((h_{+}-h_{-})(U_{T_{k}},x)-(h_{+}-h_{-})(U_{T_{k}},0)\right)}+\e^{(h_{+}-h_{-})(U_{T_{k}},0)}}{1+\e^{(h_{+}-h_{-})(U_{T_{k}},0)}}\\
 & \le\frac{1}{2}\theta x+\log\frac{\e^{A(x)}+\e^{k}}{1+\e^{k}}.
\end{align*}
In the last inequality, we have used the fact that the function $y \mapsto \log \frac{e^A + y}{1 + y}$ is decreasing for positive $y$ as long as $A > 0$.  
By taking $k$ sufficiently large, we can make this last term as small
as we like, so we conclude that, for each $x\le-M_{0}$, there is
a $C'(x)>0$ such that
\[
\sup_{k\ge C'(x)}\mathbb{P}\left(h_{\mathsf{V}}(U_{T_{k}},x)-h_{\mathsf{V}}(U_{T_{k}},0)\le0\right)\ge\frac{\eps}{2}.
\]
But by the assumed stationarity of $h_{\mathsf{V}}$, the law of $h_{\mathsf{V}}(U_{T_{k}},x)-h_{\mathsf{V}}(U_{T_{k}},0)$
does not depend on $k$, so in fact we have
\[
\mathbb{P}\left(h_{\mathsf{V}}(U_{T_{k}},x)-h_{\mathsf{V}}(U_{T_{k}},0)\le0\right)\ge\frac{\eps}{2}\qquad\text{for all }x\le-M_{0}.
\]
This contradicts the fact that $h_{\mathsf{V}}(0,\cdot)\in\mathcal{V}(\theta)$
a.s., since the latter implies that
\[
\lim_{x\to-\infty}h_{\mathsf{V}}(U_{T_{k}},x)=+\infty\qquad\text{a.s.}
\]
Therefore, we have shown \eqref{eq:hUt_tight}. A similar argument works for $h_{-}(U_{T},0)-h_{+}(U_{T},0)$, so
in fact we have
\[
\sup_{T\in(0,\infty)}\mathbb{P}\left(\left|h_{+}(U_{T},0)-h_{-}(U_{T},0)\right|>K\right)<\eps,
\]
and hence that the family of random variables $(h_{+}(U_{T},0)-h_{-}(U_{T},0))_{T}$
is tight.

Combined with \cref{eq:Videntity,eq:hplushminusconverge},
this implies that if we define
\[
J_{t}\coloneqq\frac{1}{2}\left(h_{+}(t,0)+h_{-}(t,0)\right)
\]
and 
\[
\underline{\mathbf{h}}(t,x)=\mathbf{h}(t,x)-(J_{t},J_{t}),
\]
then the family of random variables $(\underline{\mathbf{h}}(U_{T},\cdot))_{T}$
is also tight in the topology of $\mathcal{C}_{\mathrm{KPZ}}^{2}$.
Hence, there is a sequence $T_{k}\uparrow\infty$ and a measure $\psi$
on $\mathcal{C}_{\mathrm{KPZ}}^{2}$ such that 
\begin{equation}
\lim_{k\to\infty}\Law(\underline{\mathbf{h}}(U_{T_k},\cdot))=\psi\label{eq:hunderlinesconverging}
\end{equation}
weakly. Now the process $(\underline{\mathbf{h}}(t,\cdot))_{t}$ is
a Markov process with the Feller property by \cref{prop:projected-feller}. Specifically, we apply the proposition with the linear operator $g:\mathcal{C}_{\mathrm{KPZ}}^{2}\to\RR^{2}$ defined by $g[f_1,f_2] = \Bigl(\frac{1}{2}(f_2(0) - f_1(0)),\frac{1}{2}(f_2(0) - f_1(0))\Bigr)$.
Thus we can apply the Krylov--Bogoliubov theorem (see e.g.~\cite[Theorem 3.1.1]{Da-Prato-Zabczyk-1996})
to conclude that, if $\tilde{\mathbf{h}}=(\tilde{h}_{-},\tilde{h}_{+})$
is a vector of solutions to \cref{eq:KPZ} with initial data
$\tilde{\mathbf{h}}(0,\cdot)\sim\psi$, and we define
\[
\tilde{J}_{t}\coloneqq\frac{1}{2}\left(\tilde{h}_{+}(t,0)+\tilde{h}_{-}(t,0)\right)
\]
and
\[
\underline{\tilde{\mathbf{h}}}(t,x)\coloneqq\tilde{\mathbf{h}}(t,x)-(\tilde{J}_{t},\tilde{J}_t),
\]
then $\underline{\tilde{\mathbf{h}}}(t,\cdot)\sim\psi$ for each
$t\ge0$. In particular, $(\tilde{h}_{+}(t,0)-\tilde{h}_{-}(t,0))$
is a tight family of random variables. But \cref{eq:hplushminusconverge,eq:hunderlinesconverging} imply that $\underline{\tilde{\mathbf{h}}}(0,\cdot)-\underline{\tilde{\mathbf{h}}}(0,0)\sim\nu_{\theta}$,
and then \cref{thm:zeroval-stationarity} implies that the family
of random variables $(\tilde{h}_{+}(t,0)-\tilde{h}_{-}(t,0))_{t}$
is not tight, which is a contradiction.
\end{proof}
Using the tools developed in this section, we can also prove \cref{thm:whatcanyouconvergeto}.
\begin{proof}[Proof of \cref{thm:whatcanyouconvergeto}]
Let $h_{-}$ and $h_{+}$ be solutions to \cref{eq:KPZ} with
initial condition $(h_{-},h_{+})(0,\cdot)=\mathbf{A}[h_{\mathsf{V}}(0,\cdot)]$,
with $\mathbf{A}$ defined as in \cref{lem:A_map}. By \cref{eq:VAid,prop:Vsolves} similarly to as in the proof of \cref{prop:noVshaped},
we see that $h_{\mathsf{V}}(t,\cdot)=V[(h_{-},h_{+})(t,\cdot)]$ for
all $t\ge0$.
We see that $(\pi_{0}[(h_{-},h_{+})(t,\cdot)])_{t \ge 0}$
is a tight family of random variables in $\mathcal{C}_{\mathrm{KPZ}}^{2}$
by \cref{prop:spatially-homog-convergence}. Now we note that
\begin{align*}
    \pi_{0}[h_{\mathsf{V}}(t,\cdot)](x)=h_{\mathsf{V}}(t,x)-h_{\mathsf{V}}(t,0)  \overset{\cref{eq:Vdef}}&{=}\log\frac{\e^{h_{-}(t,x)}+\e^{h_{+}(t,x)}}{\e^{h_{-}(t,0)}+\e^{h_{+}(t,0)}}\\
 & =\log\frac{\e^{h_{-}(t,x)-h_{-}(t,0)}+\e^{h_{+}(t,x)-h_{+}(t,0)+h_{+}(t,0)-h_{-}(t,0)}}{1+\e^{h_{+}(t,0)-h_{-}(t,0)}}.
\end{align*}
This implies that, if we define the map $\tilde{V}\colon[0,1]\times\mathcal{C}_{\mathrm{KPZ}}^{2}\to\mathcal{C}_{\mathrm{KPZ}}$
by
\begin{equation}
    \tilde{V}[\xi,f_{-},f_{+}]\coloneqq\log(\xi\e^{f_{-}}+(1-\xi)\e^{f_{+}})\label{eq:tildeVdef}
\end{equation}
(which generalizes the map $V$ since $V=\tilde{V}[1/2,\cdot,\cdot]$),
then
\begin{equation}
    \pi_{0}[h_{\mathsf{V}}(t,\cdot)]=\tilde{V}[(1+\e^{h_{+}(t,0)-h_{-}(t,0)})^{-1},\pi_{0}[(h_{-},h_{+})(t,\cdot)]].\label{eq:tildeVreln}
\end{equation}
Now the map $\tilde{V}$ is continuous, so since $(\pi_{0}[(h_{-},h_{+})(t,\cdot)])_{t}$
is tight and $[0,1]$ is compact, we can conclude that $(\pi_{0}[h_{\mathsf{V}}(t,\cdot)])_{t}$
is tight as well, and thus complete the proof of part~\ref{enu:Vtight}.

Now we proceed to part~\ref{enu:convtomixture}. Let $U_T\sim\Uniform([0,T])$ be independent of everything else. A Krylov--Bogoliubov
argument shows that any subsequential limit $m$ of $\pi_{0}[h_{\mathsf{V}}(U_T,\cdot)]$ is an invariant measure for the spatial increments
of the KPZ equation. The ergodic decomposition theorem and the characterization
of the extremal invariant measures given in \cref{cor:classification}
mean that there is some probability measure $\eta$ on $\RR$
such that $m=\int\mu_{\rho}\,\dif\eta(\rho)$. We claim that $\eta$
is a linear combination of the point masses $\delta_{-\theta}$ and
$\delta_{\theta}$. Suppose not, so there is a probability measure
$\eta'$ and a $\kappa>0$ such that $\{\pm\theta\}\cap\supp\eta'=\varnothing$
and
\begin{equation}
\eta-\kappa\eta'\text{ is a (nonnegative) measure}.\label{eq:partof}
\end{equation}
This implies that there is an $\eps\in(0,\kappa)$ and an $M<\infty$
such that, if $|x|>M$, $(f_{-},f_{+})\sim\nu_{\theta}$, and $g\sim m$,
then
\begin{equation}
\mathbb{P}\left(\left|\frac{g(x)}{x}+\theta\right|\wedge\left|\frac{g(x)}{x}-\theta\right|>\eps\right)>\frac{\kappa}{2}\label{eq:slopesfar}
\end{equation}
and
\begin{equation}
\mathbb{P}\left(\left|\frac{f_{+}(x)}{x}-\theta\right|\vee\left|\frac{f_{-}(x)}{x}+\theta\right|>\frac{\eps}{8}\right)<\frac{\eps}{8}.\label{eq:slopesclose}
\end{equation}
Fix 
\begin{equation}
x=\frac{4\log2}{\eps},\label{eq:xchoice}
\end{equation}
assuming that $\eps$ is sufficiently small to guarantee $|x| > M$.
Now \cref{eq:lawsconvtom,eq:slopesfar} imply
that, if $k$ is chosen sufficiently large, then 
\begin{equation}
\mathbb{P}\left(\left|\frac{h_{\mathsf{V}}(U_{T_{k}},x)-h_{\mathsf{V}}(U_{T_{k}},0)}{x}-\theta\right|\wedge\left|\frac{h_{\mathsf{V}}(U_{T_{k}},-x)-h_{\mathsf{V}}(U_{T_{k}},0)}{-x}+\theta\right|>\frac{\eps}{2}\right)>\frac{\kappa}{2}.\label{eq:goodchanceofwrongslope}
\end{equation}
On the other hand, we have by \cref{prop:spatially-homog-convergence,eq:slopesclose} that, if $k$ is sufficiently large,
then
\begin{equation}
\mathbb{P}\left(\left|\frac{h_{+}(U_{T_{k}},x)-h_{+}(U_{T_{k}},0)}{x}-\theta\right|>\frac{\eps}{4}\right)<\frac{\eps}{4}<\frac{\kappa}{4}\label{eq:hplusslopebd}
\end{equation}
and
\begin{equation}
\mathbb{P}\left(\left|\frac{h_{-}(U_{T_{k}},-x)-h_{-}(U_{T_{k}},0)}{-x}+\theta\right|>\frac{\eps}{4}\right)<\frac{\eps}{4}<\frac{\kappa}{4}\label{eq:hminusslopebd}
\end{equation}
(with the latter inequalities because we assumed that $\eps<\kappa$). Now, continuing from \cref{eq:goodchanceofwrongslope}, we can
write
\begin{align}
\frac{\kappa}{2} & <\mathbb{P}\left(\left|\frac{h_{\mathsf{V}}(U_{T_{k}},x)-h_{\mathsf{V}}(U_{T_{k}},0)}{x}-\theta\right|\wedge\left|\frac{h_{\mathsf{V}}(U_{T_{k}},-x)-h_{\mathsf{V}}(U_{T_{k}},0)}{-x}+\theta\right|>\frac{\eps}{2}\right)\nonumber \\
 & \le\mathbb{P}\left(h_{+}(U_{T_{k}},0)\ge h_{-}(U_{T_{k}},0)\text{ and }\left|\frac{h_{\mathsf{V}}(U_{T_{k}},x)-h_{\mathsf{V}}(U_{T_{k}},0)}{x}-\theta\right|>\frac{\eps}{2}\right)\nonumber \\
 & \qquad+\mathbb{P}\left(h_{+}(U_{T_{k}},0)\le h_{-}(U_{T_{k}},0)\text{ and }\left|\frac{h_{\mathsf{V}}(U_{T_{k}},-x)-h_{\mathsf{V}}(U_{T_{k}},0)}{-x}+\theta\right|>\frac{\eps}{2}\right).\label{eq:splitbyzeroval}
\end{align}
If $h_{+}(U_{T_{k}},0)\ge h_{-}(U_{T_{k}},0)$ and $\left|\frac{h_{\mathsf{V}}(U_{T_{k}},x)-h_{\mathsf{V}}(U_{T_{k}},0)}{x}-\theta\right|>\eps/2$,
then by \cref{lem:Vclosetooneofthem}(\ref{enu:fplusgtr}), we
have 
\[
\left|\frac{h_{+}(U_{T_{k}},x)-h_{+}(U_{T_{k}},0)}{x}-\theta\right|>\frac{\eps}{2}-\frac{\log2}{x}\overset{\cref{eq:xchoice}}{=}\frac{\eps}{4}.
\]
The hypothesis that $(h_-,h_+)(U_{T_k},\cdot)\in\mathcal{X}(\theta)$ is satisfied by \cref{prop:Xpreserved}, since \cref{lem:A_map} tells us that $(h_-,h_+)(0,\cdot)\in\mathcal{X}(\theta)$ almost surely.
On the other hand, we similarly observe that if $h_{+}(U_{T_{k}},0)\le h_{-}(U_{T_{k}},0)$ and $\left|\frac{h_{\mathsf{V}}(U_{T_{k}},-x)-h_{\mathsf{V}}(U_{T_{k}},0)}{-x}+\theta\right|>\eps/2$,
then by \cref{lem:Vclosetooneofthem}(\ref{enu:fminusgtr}), we
have 
\[
\left|\frac{h_{-}(U_{T_{k}},-x)-h_{+}(U_{T_{k}},0)}{-x}+\theta\right|>\frac{\eps}{4}.
\]
Using these observations in \cref{eq:splitbyzeroval}, we obtain
\begin{align*}
\frac{\kappa}{2} & <\mathbb{P}\left(\left|\frac{h_{+}(U_{T_{k}},x)-h_{+}(U_{T_{k}},0)}{x}-\theta\right|>\frac{\eps}{4}\right)+\mathbb{P}\left(\left|\frac{h_{-}(U_{T_{k}},-x)-h_{+}(U_{T_{k}},0)}{-x}+\theta\right|>\frac{\eps}{4}\right)<\frac{\kappa}{2},
\end{align*}
with the last inequality by \cref{eq:hplusslopebd,eq:hminusslopebd}.
But this is a contradiction, and so the proof is complete.
\end{proof}

Finally, we prove \cref{thm:asconvergence} in a similar way.
\begin{proof}[Proof of \cref{thm:asconvergence}]
    Let $\mathbf{f} = \mathbf{A}[f_{\mathsf{V}}]$, and then let $\mathbf{h}^T = (h^T_-,h^T_+)$ solve \cref{eq:KPZ} with initial condition $\mathbf{h}^T(-T,\cdot)=\mathbf{f}$. By \cref{eq:VAid,prop:Vsolves}, this means that  $h_{\mathsf{V}}^T(0,\cdot) = V[\mathbf{h}^T(0,\cdot)]$. Defining $\tilde V$ as in \cref{eq:tildeVdef}, we have in the same way as \cref{eq:tildeVreln} that $\pi_0[h^T_{\mathsf{V}}(0,\cdot)] = \tilde V[(1+\e^{h^T_+(0,0)-h^T_-(0,0)})^{-1},\pi_0[\mathbf{h}^T(0,\cdot)]]$. By \cref{prop:spatially-homog-convergence}, we have $\lim\limits_{T\to\infty} \pi_0[\mathbf{h}^T(0,\cdot)]=\overline{\mathbf{f}}$ almost surely. Now for any sequence $T_k\uparrow\infty$, we can find a subsequence $T_{k_\ell}\uparrow\infty$ such that $\xi\coloneqq\lim\limits_{\ell\to\infty} (1+\e^{h_+^T(0,0)-h_-^T(0,0)})^{-1}$ exists, and then \cref{eq:asconvergence} follows from the continuity of $\tilde V$.
\end{proof}

\section{Fluctuations of the shock location}\label{sec:btflucts}

To complete the proof of \cref{thm:btflucts}, we need to relate
the statistics of $b_{t}$ to the statistics $h_{+}(t,0)-h_{-}(t,0)$
that have been studied in \cref{sec:zerovalflucts}. The fact
that $h_{+}(t,0)-h_{-}(t,0)$ is asymptotically linear with slope
$2\theta$ means that these quantities should be approximately related.

\subsection{Using the asymptotic slope}

The following lemma will help us make this intuition precise. In the
application, we will take $\mathcal{J}(t,x)=h_{+}(t,x)-h_{-}(t,x)$. 
\begin{lem}
\label{lem:h_to_b}Fix $\theta>0$. Let $\{\mathcal{J}(t,x)\st t\ge0,x\in\RR\}$
be a real-valued stochastic process such that the following hold.
\begin{enumerate}
\item \label{enu:cont}For each fixed $t\ge0$, with probability $1$, $x\mapsto\mathcal{J}(t,x)$
is continuous and strictly increasing.
\item \label{enu:drift}For each fixed $t\ge0$, $\lim\limits_{|x|\to\infty}\frac{\mathcal{J}(t,x)}{x}=2\theta$.
In particular, $\lim\limits_{x\to\pm\infty}\mathcal{J}(t,x)=\pm\infty$,
which together with Assumption~\ref{enu:cont} means that $x\mapsto\mathcal{J}(t,x)$
is a bijection $\RR\to\RR$.
\item \label{enu:dist}For some exponent $\alpha>0$, $t^{-\alpha}\mathcal{J}(t,0)$
converges in distribution to an almost-surely finite random variable
$Y$.
\item \label{enu:to0}Given the exponent $\alpha$ from Assumption~\ref{enu:dist},
for each $t\ge0$ and $\eps\in(0,2\theta)$, the random variable $t^{-\alpha}M_{t,\eps,\theta}$
converges to $0$ in probability, where
\begin{equation}
M_{t,\eps,\theta}\coloneqq\sup_{x\in\RR}\left[\left|\mathcal{J}(t,x)-\mathcal{J}(t,0)-2\theta x\right|-\eps|x|\right].\label{eq:Mdef}
\end{equation}
 Note that $M_{t,\eps,\theta}$ is almost-surely finite by Assumption~\ref{enu:drift}.
\end{enumerate}
Now let $b_{t}$ be the unique $x\in\RR$ such that $\mathcal{J}(t,x)=0$.
Then, as $t\to\infty$, $t^{-\alpha}b_{t}$ converges in distribution
to $-\frac{Y}{2\theta}$.
\end{lem}

\begin{proof}
Let $\eps\in(0,2\theta)$. By the definition of $M_{t,\eps,\theta}$,
we have 
\begin{align}
-M_{t,\eps,\theta}+(2\theta-\eps)x & \le\mathcal{J}(t,x)-\mathcal{J}(t,0)\le M_{t,\eps,\theta}+(2\theta+\eps)x,\qquad x\ge0;\label{eq:positivepart}\\
-M_{t,\eps,\theta}+(2\theta+\eps)x & \le\mathcal{J}(t,x)-\mathcal{J}(t,0)\le M_{t,\eps,\theta}+(2\theta-\eps)x,\qquad x\le0.\label{eq:negativepart}
\end{align}
We consider two cases. If $\mathcal{J}(t,0)<0$, then since $x\mapsto\mathcal{J}(t,x)$
is strictly increasing, we have $b_{t}>0$. By \cref{eq:positivepart},
this means that 
\[
-M_{t,\eps,\theta}+(2\theta-\eps)b_{t}\le-\mathcal{J}(t,0)\le M_{t,\eps,\theta}+(2\theta+\eps)b_{t},
\]
and so 
\[
\frac{-M_{t,\eps,\theta}-\mathcal{J}(t,0)}{2\theta+\eps}\le b_{t}\le\frac{M_{t,\eps,\theta}-\mathcal{J}(t,0)}{2\theta-\eps}.
\]
Similarly, if $\mathcal{J}(t,0)>0$, then we have
\[
\frac{-M_{t,\eps,\theta}-\mathcal{J}(t,0)}{2\theta-\eps}\le b_{t}\le\frac{M_{t,\eps,\theta}-\mathcal{J}(t,0)}{2\theta+\eps}.
\]
Thus, in either case, we have
\begin{equation}
\frac{-M_{t,\eps,\theta}}{2\theta+\eps}+\min\left\{ \frac{-\mathcal{J}(t,0)}{2\theta-\eps},\frac{-\mathcal{J}(t,0)}{2\theta+\eps}\right\} \le b_{t}\le\frac{M_{t,\eps,\theta}}{2\theta-\eps}+\max\left\{ \frac{-\mathcal{J}(t,0)}{2\theta-\eps},\frac{-\mathcal{J}(t,0)}{2\theta+\eps}\right\} .\label{eq:btbd}
\end{equation}
Now Assumption~\ref{enu:to0} states that $t^{-\alpha}M_{t,\eps,\theta}$
converges to $0$ in probability for each fixed $\eps$, and Assumption~\ref{enu:dist}
states that $t^{-\alpha}\mathcal{J}(t,0)$ converges in distribution
to $Y$. Using these assumptions in \cref{eq:btbd}, we see that
the family of random variables $(t^{-\alpha}b_{t})_{t\ge1}$ is tight,
and for each $\eps>0$, any subsequential limit $\overline{Y}$ must
be stochastically bounded above and below by $\min\left\{ \frac{-Y}{2\theta-\eps},\frac{-Y}{2\theta+\eps}\right\} $
and $\max\left\{ \frac{-Y}{2\theta-\eps},\frac{-Y}{2\theta+\eps}\right\} $,
respectively. Letting $\eps\downarrow0$, we obtain the claimed convergence
in distribution.
\end{proof}
We now use \cref{lem:h_to_b} to prove part~\ref{enu:stationaryic},
and complete the proof of part~\ref{enu:tiltedic}, of \cref{thm:btflucts}.
\begin{proof}[Proof of \cref{eq:stationary-bt,eq:shockreferenceframe-bt}]
We apply \cref{lem:h_to_b} with $\mathcal{J}(t,x)=h_{+}(t,x)-h_{-}(t,x)$.
We simply need to check the assumptions. Assumptions~\ref{enu:cont}
and~\ref{enu:drift} are verified in each case by \cref{eq:alwaysinX0}
(which holds for $\hat{\nu}_{\theta}$ as well by absolute continuity)
and \cref{prop:Xpreserved}. Assumption~\ref{enu:dist} is proved
in the two cases
by \cref{thm:zeroval-stationarity} and \cref{eq:shockreferenceframe-zerodiff}, with $\alpha=1/2$ and $Y\sim\mathcal{N}(0,2\theta)$ in both cases.

We now verify Assumption~\ref{enu:to0}. In the $\nu_{\theta}$ case,
the joint stationarity in \cref{prop:joint-stationarity}
shows that the law of $M_{t,\eps,\theta}$ does not depend on $t$, and hence
\begin{equation}
t^{-1/2}M_{t,\eps,\theta}\to0\qquad\text{in probability as }t\to\infty.\label{eq:convinprob}
\end{equation}
For the $\hat{\nu}_{\theta}$ case, we use the $\nu_{\theta}$ case
and the Cauchy--Schwarz inequality. Let $\EE$ and $\hat{\EE}$
denote expectation under which $\mathbf{h}(0,\cdot)$ is distributed
according to $\nu_{\theta}$ and $\hat{\nu}_{\theta}$, respectively,
independent from the noise. Define the Radon--Nikodym derivative
$R$ as in \cref{eq:Rdef}. Then we have, for any $\delta>0$,
that
\[
\hat{\EE}[\mathbf{1}\{M_{t,\eps,\theta}>\delta t^{1/2}\}]=\EE[\mathbf{1}\{M_{t,\eps,\theta}>\delta t^{1/2}\}R]\le\left(\EE[\mathbf{1}\{M_{t,\eps,\theta}>\delta t^{1/2}\}]\right)^{1/2}\left(\EE[R^{2}]\right)^{1/2}.
\]
The right side goes to zero as $t\to \infty$ by \cref{eq:convinprob} and
the fact that $R$ is a multiple of a Gamma-distributed random variable (as noted after \cref{eq:Rdef}), which has finite second moment. 
\end{proof}

\subsection{Flat initial data}

The proof of \cref{thm:btflucts}(\ref{enu:flatic}) is more
technical than \cref{thm:btflucts}(\ref{enu:stationaryic}), since understanding the dependence of the law of $M_{t,\eps,\theta}$
on $t$ is much less trivial. We make the definition
\[
\mathcal{H}_{t}(x\viiva y)=\frac{\log Z(t,t^{2/3}x\viiva0,t^{2/3}y)+\frac{t}{24}}{t^{1/3}}
\]
and set $\mathcal{H}_{t}(x)=\mathcal{H}_{t}(x\viiva0)$. We also define
$H_{t}(x\viiva y)=\mathcal{H}_{t}(x\viiva y)+\frac{(x-y)^{2}}{2}$ and
$H_{t}(x)=H_{t}(x\viiva0)=\mathcal{H}_{t}(x)+\frac{x^{2}}{2}$. We first
recall a tail bound on the one-point statistics of $\mathcal{H}_{t}(x\viiva y)$.
\begin{lem}
\label{lem:onept}There exist constants $C<\infty$ and $c>0$ such
that, for all $y\in\RR$, $t\ge1$, $x,y\in\RR$, and $m\ge0$, we
have
\begin{equation}
\mathbb{P}\left(\left|H_{t}(x\viiva y)\right|>m\right)\le C\e^{-cm^{3/2}}.\label{eq:tailbound-anywhere}
\end{equation}
\end{lem}

\begin{proof}
By \cite[Theorem 1.11]{Corwin-Ghosal-2020} and \cite[Theorem 1.1]{Corwin-Ghosal-2020b},
we have constants $C<\infty$ and $c>0$ such that
\begin{equation}
\mathbb{P}\left(|\mathcal{H}_{t}(0)|>m\right)\le C\e^{-cm^{3/2}}\label{eq:tailbound-singlepoint}
\end{equation}
for all $m\ge0$. (In fact, the lower tail of $\mathcal{H}_{t}(0)$
is steeper, but we do not need this. ) Using the translation invariance
\cref{eq:Z_shift} and shear invariance \cref{eq:Z_shear}
of $Z$, we see that
\[
\mathcal{H}_{t}(0)\overset{\mathrm{law}}{=}\mathcal{H}_{t}(x\viiva y)+\frac{(x-y)^{2}}{2}=H_{t}(x\viiva y)\qquad\text{for all }x,y\in\RR,
\]
and hence \cref{eq:tailbound-singlepoint} becomes \cref{eq:tailbound-anywhere}.
\end{proof}
We will need the following result on the increments of $H_{t}(x)$. 
\begin{lem}[\cite{Corwin-Ghosal-Hammond-21}, Theorem 1.3]
\label{lem:CGH_sup}There exist constants $c>0$ and $C<\infty$
such that, for all $y\in\RR$, $t\ge1$, $m\ge0$, and $\eps\in(0,1]$,
we have
\[
\mathbb{P}\left(\sup_{x\in[y,y+\eps]}\left|H_{t}(x)-H_{t}(y)\right|\ge\eps^{1/2}m\right)\le C\e^{-cm^{3/2}}.
\]
\end{lem}

The preceding two lemmas combined with a chaining argument will let
us establish the following lemma on the maximum of the KPZ solution
on a compact domain.
\begin{lem}
\label{lem:KPZmax}Let $h$ solve the KPZ equation \cref{eq:KPZ}
with $h(0,\cdot)\equiv0$. Then, for each compact set $K\subseteq\RR$,
there exist constants $C<\infty$ and $c>0$ such that for all $t>1$
and $m\ge0$, we have
\begin{equation}
\mathbb{P}\left(\sup_{x\in K}\left|\frac{h(t,t^{2/3}x)+\frac{t}{24}}{t^{1/3}}\right|\ge m\right)\le C\e^{-cm^{3/4}}.\label{eq:htailbd}
\end{equation}
\end{lem}

\begin{proof}
  It clearly suffices to consider the case when $K$ is an interval of integer length.
The proof proceeds in several steps.
\begin{thmstepnv}
\item By definition, we have $h(t,t^{2/3}x)=\log\int_{\RR}Z(t,t^{2/3}x\viiva0,y)\,\dif y$,
so after a change of variables, we get
\begin{equation}
\frac{h(t,t^{2/3}x)+\frac{t}{24}}{t^{1/3}}=t^{-1/3}\log t^{2/3}+t^{-1/3}\log\int_{\RR}\e^{t^{1/3}\mathcal{H}_{t}(x\viiva y)}\,\dif y.\label{eq:hintermsoflogint}
\end{equation}
The first term on the right side goes to $0$ as $t\to\infty$, so
it suffices to obtain tail bounds on the random variable $t^{-1/3}\log\int_{\RR}\e^{t^{1/3}\mathcal{H}_{t}(x\viiva y)}\,\dif y$.
\item We claim that it suffices to show that there exist constants $C<\infty$
and $c>0$ such that, for $t\ge1$ and $m\ge0$, we have
\begin{equation}
\mathbb{P}\left(\sup_{x\in K,y\in\RR}\frac{\left|H_{t}(x\viiva y)\right|}{(|y|+1)^{2/3}}\ge m\right)\le C\e^{-cm^{3/2}}.\label{eq:bd_frak}
\end{equation}
First, we assume \cref{eq:bd_frak} and show how it implies \cref{eq:htailbd}.
Then we will prove \cref{eq:bd_frak} in Step~\ref{step:bdfrak}
below. Assume that for some $s\ge0$, the event in \cref{eq:bd_frak}
fails, i.e. for some $m\ge0$, we have
\begin{equation}
\sup_{x\in K,y\in\RR}\frac{\left|H_{t}(x\viiva y)\right|}{(|y|+1)^{2/3}}\le m.\label{eq:les}
\end{equation}
We will show, given \cref{eq:les}, there exist constants $C_{1},C_{2}<\infty$,
independent of $m$, such that
\begin{equation}
\sup_{x\in K}\left|t^{-1/3}\log\int_{\RR}\e^{t^{1/3}\mathcal{H}_{t}(x\viiva y)}\,\dif y\right|\le C_{1}m^{2}+C_{2},\label{eq:m2pC2bd}
\end{equation}
which will imply \cref{eq:htailbd} by \cref{eq:hintermsoflogint}.
Thus we now prove \cref{eq:m2pC2bd} assuming \cref{eq:les}.
We first note that there is a constant $A>0$ such that for all $x\in K$
and $y\in\RR$, we have
\[
(|y| + 1)^{2/3}\le|x-y|+A.
\]
Then, since we are assuming that \cref{eq:les} holds, we see
that
\[
\left|H_{t}(x\viiva y)\right|\le m(|x-y|+A)\qquad\text{for all }x\in K,y\in\RR.
\]
Then, for $x\in K$, we obtain the upper bound
\begin{align*}
t^{-1/3}\log\int_{\RR}\e^{t^{1/3}\mathcal{H}_{t}(x\viiva y)}\,\dif y & \le mA+t^{-1/3}\log\int_{\RR}\exp\left\{ -t^{1/3}\left(\frac{(x-y)^{2}}{2}-m|x-y|\right)\right\} \,\dif y\\
 & =mA+t^{-1/3}\log\left(2\int_{0}^{\infty}\exp\left\{ -t^{1/3}\left(\frac{y^{2}}{2}-my\right)\right\} \,\dif y\right)\\
 & \le mA+\frac{m^{2}}{2}+t^{-1/3}\log\left(2\int_{\RR}\exp\left\{ -t^{1/3}y^{2}/2\right\} \,\dif y\right),
\end{align*}
and the last term on the right side is independent of $m$ and goes
to $0$ as $t\to\infty$. Furthermore, for $x\in[0,1]$ and $m>t^{-1/6}$,
we have the lower bound
\begin{align*}
t^{-1/3}\log\int_{\RR}\e^{t^{1/3}\mathcal{H}_{t}(x\viiva y)}\,\dif y & \ge-mA+t^{-1/3}\log\int_{\RR}\exp\left\{ -t^{1/3}\left(\frac{(x-y)^{2}}{2}+m|x-y|\right)\right\} \,\dif y\\
 & =-mA+t^{-1/3}\log\left(2\int_{0}^{\infty}\exp\left\{ -t^{1/3}\left(\frac{y^{2}}{2}+my\right)\,\dif y\right\} \right)\\
 & =-mA+\frac{m^{2}}{2}+t^{-1/3}\log(2t^{-1/6})+t^{-1/3}\log\int_{mt^{1/6}}^{\infty}\e^{-y^{2}/2}\,\dif y\\
 & \ge-mA+t^{-1/3}\log\left(\frac{m^{2}t^{1/3}-1}{m^{3}t^{1/2}}\right)-1\\
 & \ge-mA-\frac{m^{3}t^{1/6}}{m^{2}t^{1/3}-1}-1,
\end{align*}
and this is greater than $-C_{1}m^{2}-C_{2}$ for constants $C_{1},C_{2}<\infty$,
which completes the proof of \cref{eq:m2pC2bd}. In the penultimate
step, we used the standard Gaussian tail bound (see e.g. \cite[Theorem 1.2.6]{Durrett})
\[
\int_{z}^{\infty}\e^{-y^{2}/2}\,\dif y\ge(z^{-1}-z^{-3})\e^{-z^{2}/2},
\]
and in the last step, we used the bound $\log z\ge-z^{-1}$.
\item \label{step:bdfrak}Now we prove \cref{eq:bd_frak}. First, let
$\eps\in(0,1]$, and assume that 
\begin{equation}
|y_{1}-y_{2}|\vee|x_{1}-x_{2}|\le\eps.\label{eq:y1y2x1x2assumption}
\end{equation}
Then we have
\begin{align}
\mathbb{P} & \left(\left|H_{t}(x_{2}\viiva y_{2})-H_{t}(x_{1}\viiva y_{1})\right|\ge m\eps^{1/2}\right)\nonumber \\
 & \le\mathbb{P}\left(\left|H_{t}(x_{2}\viiva y_{2})-H_{t}(x_{1}\viiva y_{2})\right|\ge\frac{m\eps^{1/2}}{2}\right)+\mathbb{P}\left(\left|H_{t}(x_{1}\viiva y_{2})-H_{t}(x_{1}\viiva y_{1})\right|\ge\frac{m\eps^{1/2}}{2}\right)\nonumber \\
 & =\mathbb{P}\left(\left|H_{t}(x_{2}-y_{2})-H_{t}(x_{1}-y_{2})\right|\ge\frac{m\eps^{1/2}}{2}\right)+\mathbb{P}\left(\left|H_{t}(x_{1}-y_{2})-H_{t}(x_{1}-y_{1})\right|\ge\frac{m\eps^{1/2}}{2}\right)\nonumber\\&\le C\e^{-cm^{3/2}}\label{eq:setupforchaining}
\end{align}
for constants $C<\infty$ and $c>0$. In the first inequality we used
a union bound, in the identity we used translation-invariance, and
in the last inequality we used \cref{lem:CGH_sup} twice.

Now, for $b\in\mathbb{N}$, we partition the rectangle $K\times[-b,b]$ into $N(b)\coloneqq 2|K|b$ squares
of side length $1$, enumerated as $S_{1},\ldots,S_{N(b)}$, and let
$(x_{i},y_{i})$ be the center point of $S_{i}$.%
For each $i$, the bound \cref{eq:setupforchaining} implies
that the assumptions of \cref{lem:chaining} hold with $d=2$,
$\alpha_{i}=1/2$, $\beta_{i}=3/2$, $r_{i}=1$, and $T=S_{i}$, and
so we obtain constants $C<\infty$ and $c>0$ (independent of $i$,
$t$, and $b$) such that, for each $m\ge0$, we have
\begin{align}
\mathbb{P} & \left(\sup_{\substack{((x_{1},x_{2}),(y_{1},y_{2}))\in S_{i}^{2}}
}\left[\frac{\left|H_{t}(x_{2}\viiva y_{2})-H_{t}(x_{1}\viiva y_{1})\right|}{g(|y_{2}-y_{1}|)+g(|x_{2}-x_{1}|)}\right]\ge m\right)\le C\e^{-cm^{3/2}},\label{eq:sIbd}
\end{align}
where we have defined the nonnegative continuous function
\[
g(z)=\begin{cases}
z^{1/2}\left(\log\frac{2}{z}\right)^{2/3}, & z\in(0,1];\\
0, & z=0.
\end{cases}
\]

Now, if we let
\[
A\coloneqq1+2\sup_{z\in(0,1]}g(z).
\]
then we obtain using \cref{eq:sIbd,lem:onept}
that
\begin{align*}
\mathbb{P} & \left(\sup_{x\in K,y\in[-b,b]}\left|H_{t}(x\viiva y)\right|\ge Am\right)\\
 & \le\sum_{i=1}^{N(b)}\left(\mathbb{P}\left(\sup_{(x,y)\in S_{i}}\frac{\left|H_{t}(x\viiva y)-H_{t}(x_{i}\viiva y_{i})\right|}{g(|y-y_{i}|)+g(|x-x_{i}|)+1}\ge\frac{m}{2}\right)+\mathbb{P}\left(\left|H_{t}(x_{i}\viiva y_{i})\right|\ge\frac{m}{2}\right)\right)\\
 & \le Cb\e^{-cm^{3/2}}
\end{align*}
for new constants $C<\infty$ and $c>0$ that do not depend on $b$
or $t$. Then we obtain
\begin{align*}
\mathbb{P}\left(\sup_{x\in K,y\in\RR}\frac{\left|H_{t}(x\viiva y)\right|}{A\left(|y|+1\right)^{2/3}}\ge m\right) & \le\sum_{b=1}^{\infty}\mathbb{P}\left(\sup_{x\in K,b-1\le|y|<b}|H_{t}(x\viiva y)|\ge Am\left(b+1\right)^{2/3}\right)\\
 & \le\sum_{b=1}^{\infty}Cb\e^{-cm^{3/2}(b+1)}\le C'\e^{-c'm^{3/2}}
\end{align*}
for new constants $C',c'>0$. This completes the proof of \cref{eq:bd_frak}.\qedhere
\end{thmstepnv}
\end{proof}
The following lemma is the key to checking Assumption~\ref{enu:to0}
in \cref{lem:h_to_b}.
\begin{lem}
\label{lem:assumption4key}Let $h$ solve the KPZ equation \cref{eq:KPZ}
with $h(0,\cdot)\equiv0$. Then, for any $\eps>0$, we have the convergence
\[
t^{-1/3}\sup_{x\in\RR}\left[|h(t,x)-h(t,0)|-\eps|x|\right]\to0\qquad\text{in probability as }t\to\infty.
\]
\end{lem}

\begin{proof}
By the spatial reflection invariance \cref{eq:Z_reflect}, it
suffices to prove that
\[
t^{-1/3}\sup_{x\ge0}\left[|h(t,x)-h(t,0)|-\eps x\right]\to0\qquad\text{in probability as }t\to\infty.
\]
We write
\begin{align}
t^{-1/3}\sup_{x\ge0}\left[|h(t,x)-h(t,0)|-\eps x\right] & =\sup_{x\ge0}\left[\left|\frac{h(t,t^{2/3}(t^{-2/3}x))-h(t,0)}{t^{1/3}}\right|-\eps t^{-1/3}x\right]\nonumber \\
 & =\sup_{y\ge0}\left[\left|\frac{h(t,t^{2/3}y)-h(t,0)}{t^{1/3}}\right|-\eps t^{1/3}y\right].\label{eq:sup_re}
\end{align}
Choose an integer $K>\eps^{-1}$. Note that the supremum in \cref{eq:sup_re}
is nonnegative because the quantity is $0$ when $y=0$. Hence, for
the supremum in \cref{eq:sup_re} to not be obtained in $[0,k]$,
the supremum over $y\in[k,\infty)$ must be positive. Then, for $\delta>0$,
we have
\begin{align}
\mathbb{P} & \left(t^{-1/3}\sup_{x\ge0}\left[|h(t,x)-h(t,0)|-\eps x\right]>\delta\right)\nonumber \\
 & \le\mathbb{P}\left(\left|h(t,0)+\frac{t}{24}\right|>t^{2/3}\right)+\mathbb{P}\left(\sup_{y\in[0,k]}\left[\left|\frac{h(t,t^{2/3}y)-h(t,0)}{t^{1/3}}\right|-\eps t^{1/3}y\right]>\delta\right)\nonumber \\
 & \qquad+\sum_{i=k}^{\infty}\mathbb{P}\left(\sup_{y\in[i,i+1]}\left|\frac{h(t,t^{2/3}y)+\frac{t}{24}}{t^{1/3}}\right|>t^{1/3}(\eps i-1)\right).\label{eq:boundprobgtdelta}
\end{align}

We consider each of the terms on the right side of \cref{eq:boundprobgtdelta}
in turn. The first term goes to $0$ because of the convergence in
law of $t^{-1/3}(h(t,0)+\frac{t}{24})$ to a Tracy--Widom GOE random
variable. (See \cref{lem:TW_conv}.)

The second term of \cref{eq:boundprobgtdelta} goes to zero by
the convergence of the KPZ equation to the KPZ fixed point uniformly
on compact sets (see \cref{prop:KPZ_eq_conv}). Specifically,
we can couple $h$ to the KPZ fixed point $\mathfrak{h}$ started
from $\mathfrak{h}(0,\cdot)\equiv0$ such that, with probability $1$,
\begin{align*}
\adjustlimits\limsup_{t\to\infty}\sup_{y\in[0,k]} & \left[\left|\frac{h(t,t^{2/3}y)-h(t,0)}{t^{1/3}}\right|-\eps t^{1/3}y\right]\\
 & \le\adjustlimits\lim_{T\to\infty}\limsup_{t\to\infty}\sup_{y\in[0,k]}\left[\left|\frac{h(t,t^{2/3}y)-h(t,0)}{t^{1/3}}\right|-\eps T^{1/3}y\right]\\
 & =\adjustlimits\lim_{T\to\infty}\sup_{y\in[0,k]}\left[\left|2^{-1/3}\mathfrak{h}(0,2^{-1/3}y)-2^{-1/3}\mathfrak{h}(t,0)\right|-\eps T^{1/3}y\right]=0,
\end{align*}
where in the last step we used the continuity of the process $\mathfrak{h}$.

For the third term of \cref{eq:boundprobgtdelta}, we use\textbf{
}\cref{lem:KPZmax}, along with the spatial homogeneity of $h$,
to show that there is a constant $C$ such that
\[
\sup_{i\in\RR}\EE\left(\sup_{y\in[i,i+1]}\left|\frac{h(t,t^{2/3}y)+\frac{t}{24}}{t^{1/3}}\right|\right)^{2}\le C.
\]
This means that
\[
\sum_{i=k}^{\infty}\mathbb{P}\left(\sup_{y\in[i,i+1]}\left|\frac{h(t,t^{2/3}y)+\frac{t}{24}}{t^{1/3}}\right|>t^{1/3}(\eps i-1)\right)\le\frac{C}{t^{2/3}}\sum_{i=k}^{\infty}\frac{1}{(\eps i-1)^{2}}\to0\qquad\text{as }t\to\infty.
\]
This completes the proof.
\end{proof}
We have now assembled all of the necessary ingredients to prove \cref{thm:btflucts}(\ref{enu:flatic}).
\begin{proof}[Proof of \cref{thm:btflucts}\textup{(\ref{enu:flatic})}.]
We again use the general framework of \cref{lem:h_to_b} applied
to $\mathcal{J}(t,x)=h_{+}(t,x)-h_{-}(t,x)$, and we have to check
the assumptions. Assumptions~\ref{enu:cont} and~\ref{enu:drift}
are direct consequences of \cref{prop:Xpreserved}. We take $\alpha=1/3$
and $Y=\frac{X_{1}-X_{2}}{2}$, where $X_{1}$ and $X_{2}$ are independent
Tracy--Widom GOE random variables. Then \cref{thm:zeroval-flat}
implies that $t^{-1/3}\mathcal{J}(t,0)$ converges in distribution
to $Y$ as $t\to\infty$, and so Assumption~\ref{enu:dist} is satisfied
with these choices. To check Assumption~\ref{enu:to0}, we note
that
\begin{align*}
t^{-1/3}\sup_{x\in\RR}\left[\left|\mathcal{J}(t,x)-\mathcal{J}(t,0)-2\theta x \right|-\eps|x|\right] & \le t^{-1/3}\sup_{x\in\RR}\left[\left|h_{+}(t,x)-h_{+}(t,0)-\theta x \right|-\frac{\eps}{2}|x|\right]\\
 & \quad+t^{-1/3}\sup_{x\in\RR}\left[\left|h_{-}(t,x)-h_{-}(t,0)-\theta x \right|-\frac{\eps}{2}|x|\right],
\end{align*}
so it suffices to show the convergence of each of the two terms on
the right to zero in probability. We prove the first, as the second
is symmetrical. By the shear-invariance \cref{eq:h_shear}, we
have
\[
\sup_{x\in\RR}\left[\left|h_{+}(t,x)-h_{+}(t,0)-\theta x \right|-\frac{\eps}{2}|x|\right]\overset{\mathrm{law}}{=}\sup_{x\in\RR}\left[\left|h_{0}(t,x)-h_{0}(t,0)\right|-\frac{\eps}{2}|x|\right],
\]
and then \cref{lem:assumption4key} implies that Assumption~\ref{enu:to0}
holds. With the assumptions verified, \cref{lem:h_to_b} implies \cref{eq:flat-bt}
and the proof is complete.
\end{proof}

\appendix

\section{Technical lemmas}

Here we state a few technical lemmas that are useful at various points
in our arguments. The following chaining result is due to Dauvergne
and Virág; for simplicity, we state a version somewhat specialized
to our needs.
\begin{lem}[{\cite[Lemma 3.3]{Dauvergne-Virag-18}}]
\label{lem:chaining}Let $T=I_{1}\times\cdots\times I_{d}$ be a
product of bounded real intervals of lengths $b_{1},\ldots,b_{d}>0$.
Let $\mathcal{H}\colon T\to\RR$ be a random continuous function.
Assume that, there are constants $C<\infty$ and $c>0$ such that
for every $i\in\{1,\ldots,d\}$, there exist $\alpha_{i}\in(0,1)$,
$\beta_{i},r_{i}>0$ such that
\[
\mathbb{P}\left(\left|\mathcal{H}(x+e_{i}u)-\mathcal{H}(x)\right|\ge mu^{\alpha_{i}}\right)\le C\e^{-cm^{\beta_{i}}}
\]
for every unit coordinate vector $e_{i}$, every $m\ge0$, and every
$x,x+ue_{i}\in T$ with $u\in(0,r_{i}]$. Set $\beta=\min_{i}\beta_{i}$,
$\alpha=\max_{i}\alpha_{i}$, and $r=\max_{i}r_{i}^{\alpha_{i}}$.
Then we have
\begin{align*}
\mathbb{P}&\left(\sup\left\{ \frac{|\mathcal{H}(x+y)-\mathcal{H}(x)|}{\sum_{i=1}^{d}|y_{i}|^{\alpha_{i}}\left(\log\left(\frac{2r^{1/\alpha_{i}}}{|y_{i}|}\right)\right)^{1/\beta_{i}}}\st\begin{array}{c}
x,y+x\in T\text{ and}\\
0<|y_{i}|\le r_{i}\text{ for }1\le i\le d
\end{array}\right\} \ge m\right)\\&\qquad\le CC_{0}\e^{-c_{1}m^{\beta}}\prod_{i=1}^{d}\frac{b_{i}}{r_{i}}
\end{align*}
for constants $C_{0}<\infty$ and $c_{1}>0$ depending only on $\alpha_{1},\ldots,\alpha_{d},\beta_{1},\ldots,\beta_{d},d,c$,
and in particular not on $b_{1},\ldots,b_{d},C,r_{1},\ldots,r_{d}$.
\end{lem}

We also use the following simple lemma.
\begin{lem}
\label{lem:R1moments}Let $B$ be a two-sided Brownian motion, with
arbitrary diffusivity, and let $\alpha,\lambda>0$. Then we have
\[
\EE\left[\left(\int_{-1}^{0}\e^{B(y)+\lambda y}\,\dif y\right)^{-\alpha}\right]<\infty.
\]
\end{lem}

\begin{proof}
For $z>0$, if $\min\limits_{y\in[-1,0]}B(y)>-z$, then $\int_{-1}^{0}\e^{B(y)+\lambda y}\,\dif y\ge c\e^{-z}$,
where $c=\frac{1-\e^{-\lambda}}{\lambda}>0$. Hence, for $x>1/c$,
we can estimate
\begin{align*}
\mathbb{P}\left(\left(\int_{-1}^{0}\e^{B(y)+\lambda y}\,\dif y\right)^{-1}>x\right) & \le\mathbb{P}\left(\min_{-1\le y\le0}B(y)\le-\log(cx)\right)\\
 & =\mathbb{P}\left(|B(-1)|>\log(cx)\right)\le\frac{2\e^{-(\log cx))^{2}}}{\log cx},
\end{align*}
where the last step follows by standard Gaussian tail bounds. (See
e.g. \cite[Theorem 1.2.6]{Durrett}.) We see that the right side is
smaller than any positive power of $x^{-1}$. In particular, all of
the positive moments of the random variable $\left(\int_{-1}^{0}\e^{B(y)+\lambda y}\,\dif y\right)^{-1}$
are finite.
\end{proof}
\printbibliography[heading=bibintoc]

\end{document}